\documentclass[11pt]{amsart}
\usepackage{amsmath}
\usepackage{bbding}
\usepackage{tikz}
\usepackage{amssymb}
\usepackage{mathrsfs}

\usepackage[plainpages=false,hypertexnames=false,colorlinks=true,linkcolor=blue,citecolor=blue,urlcolor=blue]{hyperref}

\theoremstyle{plain}
\newtheorem{theorem}{Theorem}[section]
\newtheorem{lemma}[theorem]{Lemma}
\newtheorem{prop}[theorem]{Proposition}
\newtheorem{corollary}[theorem]{Corollary}

\theoremstyle{definition}

\numberwithin{equation}{section}

\def\be{\begin{equation}}
\def\ee{\end{equation}}

\begin{document}

\title[Mixed Local-Nonlocal Operator]
{Asymptotic Expansions of the Fundamental Solution for a Mixed Local--Nonlocal Operator}

\author[C.~Li]{Congming Li}
\address{School of Mathematical Sciences\\
Shanghai Jiao Tong University\\
Shanghai, 200240, China}  \email{congming.li@sjtu.edu.cn}
\author[L.~Wang]{Ling Wang}
\address{Department of Decision Sciences and BIDSA\\ Bocconi University\\ Milano, Italy}
\email{ling.wang@unibocconi.it}
\author[J.~Xie]{Jiongduo Xie\textsuperscript{*}}
\thanks{\textsuperscript{*}Corresponding author.}
\address{School of Mathematical Sciences\\
Shanghai Jiao Tong University\\
Shanghai, 200240, China}
\email{jiongduoxie@outlook.com}
\subjclass[2020]{Primary 35A08; Secondary 35C20, 35R11, 35B50, 35A21}

\keywords{Mixed local--nonlocal operators, fundamental solutions, asymptotic expansions, B\^ocher-type theorem, maximum principles}

\begin{abstract}
In this paper, we study the fundamental solution of the mixed local--nonlocal operator
$
-\Delta+(-\Delta)^s,$ $ 0<s<1.
$
We derive precise asymptotic expansions of the fundamental solution and its gradient, up to the first nontrivial correction term, both near the origin and at infinity. The expansions describe the transition between the local and nonlocal diffusion scales and include the critical regimes in which logarithmic terms occur. As applications of these asymptotics, we establish a distributional B\^ocher-type theorem for nonnegative supersolutions with an isolated singularity. We further obtain quantitative positive and antisymmetric maximum principles on punctured balls.
\end{abstract}

\maketitle

\section{Introduction}\label{sec-Intro}

Mixed local--nonlocal operators arise when diffusion mechanisms of different orders act simultaneously. In this paper, we consider the operator
\[
\mathscr{L}_s:=-\Delta+(-\Delta)^s,
\qquad 0<s<1.
\]
At the probabilistic level, \(-\mathscr{L}_s\) is the generator of the sum of a Brownian motion and an independent symmetric \(2s\)-stable process. This connects the operator with the potential theory of stable processes and mixed Brownian--stable processes. We refer to \cite{BG1960,SV2007,CKSV2010,CKS2011} for representative results in this direction. Local diffusion coupled with jump processes also appears naturally in stochastic control and jump-diffusion problems \cite{BJK2010,BK2023,DRZ2021}, as well as in population dynamics and ecological dispersal models \cite{DV2021,DPV2023}.

From the PDE point of view, equations containing both differential and integral terms belong to the broader theory of elliptic integro-differential equations developed in \cite{GL1984,BI2008,FR2024}. 
A substantial literature has recently developed on mixed local--nonlocal
elliptic operators, including regularity and maximum principles
\cite{BDVV2022,DM2024,GK2022,GL2023,SVWZ2022,SVWZ2025}, 
spectral and variational problems
\cite{DPV2022,BDVV2023,Bdd2022,HS2026}, and existence, symmetry, and
critical phenomena for nonlinear elliptic equations
\cite{BVDV2021,BMV2024,BDQ2025,GLX2026,DSVZ2025,SX2025}.

Throughout the paper, the fractional Laplacian is normalized by
\begin{equation*}
(-\Delta)^s u(x)
=
C_{n,s}\,\mathrm{PV}
\int_{\mathbb R^n}
\frac{u(x)-u(y)}{|x-y|^{n+2s}}\,\mathrm{d}y
=
C_{n,s}\lim_{\varepsilon\searrow0}
\int_{\mathbb R^n\setminus B_\varepsilon(x)}
\frac{u(x)-u(y)}{|x-y|^{n+2s}}\,\mathrm{d}y,
\end{equation*}
initially for sufficiently regular functions, where
\[
C_{n,s}
=
\frac{2^{2s}\Gamma(s+n/2)}
{\pi^{n/2}|\Gamma(-s)|}.
\]
We refer to \cite{S2007,CLM2020,FR2024} for standard background on the fractional Laplacian and elliptic integro-differential operators.

A fundamental feature of \(\mathscr{L}_s\) is its lack of homogeneity under dilations. Its Fourier symbol is
$|\xi|^2+|\xi|^{2s}.$
For large frequencies, the second-order term is dominant, while for small frequencies the fractional term is dominant. More precisely,
\[
|\xi|^2+|\xi|^{2s}\sim|\xi|^2
\qquad\text{as }|\xi|\to\infty,
\]
whereas
\[
|\xi|^2+|\xi|^{2s}\sim|\xi|^{2s}
\qquad\text{as }|\xi|\to0.
\]
This frequency transition is reflected directly in the fundamental solution. Its singularity near the pole is governed by leading order by the classical Laplacian, while its behavior at infinity is governed by leading order by the fractional Laplacian. The interaction between the two components first appears in the next term of the corresponding asymptotic expansions.

Fundamental solutions and Green kernels play a central role in potential estimates, isolated-singularity problems, and Liouville-type questions. For homogeneous local and fractional operators, the leading kernels are explicit. Green functions and fundamental solutions for various nonlocal and nonlinear nonlocal operators have been studied in \cite{AJS2018,FQ2011,DQ2026}. For sums of Brownian and stable generators, sharp Green function and heat kernel estimates are available in \cite{CKSV2010,CKS2011}. For \(n\ge 3\), the exact leading-order asymptotics of the fundamental solution follow from the general results of Rao, Song, and
Vondra\v{c}ek \cite{RaoSongVondracek2006} for subordinate Brownian
motions. Indeed, applying their Theorems to the Laplace
exponent \(\phi(\lambda)=\lambda+\lambda^s\) yields the Newtonian
leading term near the origin and the fractional Riesz leading term at
infinity. More recently, sharp two-sided estimates for the fundamental solution associated with the mixed operator were also established and 
used in the analysis of mixed Lane--Emden equations in \cite{GLX2026}.

The first purpose of this paper is to obtain information beyond such leading-order estimates. We determine the first nontrivial correction to the fundamental solution of \(\mathscr{L}_s\) both near the origin and at infinity. We obtain the corresponding expansion for its gradient at the same level of precision. The structure of the correction depends on both the dimension and the fractional exponent. In particular, critical homogeneous degrees lead to logarithmic terms. Near the origin, this occurs when \(n=3\) and \(s=\frac12\). At infinity, \(s=\frac12\) is again exceptional. The second term in the low-frequency expansion is the constant multiplier \(-1\). After multiplication by the cutoff, its inverse Fourier transform is a Schwartz function, so the following \(|\xi|\)-term gives the first nonzero correction in the far field.

We work mainly with distributional solutions. Set
\[
\mathcal L^{2s}(\mathbb R^n)
=
\left\{
u\in L^1_{\mathrm{loc}}(\mathbb R^n)
\,\middle|\,
\int_{\mathbb R^n}
\frac{|u(x)|}{1+|x|^{n+2s}}\,\mathrm{d}x<\infty
\right\}.
\]
If \(\Omega\subset\mathbb R^n\) is open and
\(u\in\mathcal L^{2s}(\mathbb R^n)\), we define
\(\mathscr{L}_su\) in \(\Omega\) by
\[
\langle\mathscr{L}_su,\phi\rangle
=
\int_{\mathbb R^n}
u(x)\mathscr{L}_s\phi(x)\,\mathrm{d}x,
\qquad
\phi\in C_c^\infty(\Omega).
\]
Since
$|(-\Delta)^s\phi(x)|
\le
C_\phi(1+|x|)^{-n-2s}$,
the pairing is well defined.

We now state the first main result. The Fourier transform is normalized by
\begin{equation}
\widehat f(\xi)
=
\int_{\mathbb R^n}e^{-ix\cdot\xi}f(x)\,\mathrm{d}x,
\qquad
f(x)
=
\frac{1}{(2\pi)^n}
\int_{\mathbb R^n}
e^{ix\cdot\xi}\widehat f(\xi)\,\mathrm{d}\xi.
\label{eq:fourier-normalization}
\end{equation}
We write
\[
\omega_{n-1}
=
|\mathbb S^{n-1}|
=
\frac{2\pi^{n/2}}{\Gamma(n/2)}
\quad
\text{and}
\quad
\kappa_{n,\alpha}
=
\frac{\Gamma\!\left(\frac{n-\alpha}{2}\right)}
{2^\alpha\pi^{n/2}\Gamma\!\left(\frac{\alpha}{2}\right)}.
\]
For
\(\alpha\notin\{n,n+2,n+4,\ldots\}\), the constant
\(\kappa_{n,\alpha}\) is understood through the meromorphic continuation
of the Riesz-kernel family described in
Appendix~\ref{app:homogeneous-distributions}.

\begin{theorem}\label{thm:fundamental-solution}
Let \(n\geq2\) and \(0<s<1\). Define
\[
    \Gamma_{n,s}
    =
    \mathcal F^{-1}
    \left(\frac{1}{|\xi|^2+|\xi|^{2s}}\right)
    \in\mathcal S'(\mathbb R^n).
\]
Then
\begin{equation}
    \mathscr{L}_s\Gamma_{n,s}
    =\delta_0
    \qquad\text{in }\mathcal S'(\mathbb R^n),
    \label{eq:fundamental-solution-equation}
\end{equation}
and \(\Gamma_{n,s}\) is represented by a smooth real-valued radial
function in \(\mathbb R^n\setminus\{0\}\). Moreover, the following assertions hold.

\begin{enumerate}
\item \textbf{Asymptotics at the origin.}

\begin{description}
\item[\(n\geq3\) and \(4-2s<n\)]
As \(|x|\to0\),
\begin{equation}
\begin{aligned}
    \Gamma_{n,s}(x)
    &=
    \frac{|x|^{2-n}}{(n-2)\omega_{n-1}}
    -\kappa_{n,4-2s}|x|^{4-2s-n}
    +o\!\left(|x|^{4-2s-n}\right),\\
    \nabla\Gamma_{n,s}(x)
    &=
    -\frac{x}{\omega_{n-1}|x|^n}
    -(4-2s-n)\kappa_{n,4-2s}
      x|x|^{2-2s-n}
    +o\!\left(|x|^{3-2s-n}\right).
\end{aligned}
\label{eq:origin-generic-expansion}
\end{equation}

\item[\(n=3\) and \(s=\frac12\)]
As \(|x|\to0\),
\begin{equation}
\begin{aligned}
    \Gamma_{3,1/2}(x)
    &=
    \frac{1}{4\pi|x|}
    -\frac{1}{2\pi^2}\log\frac1{|x|}
    +O(1),\\
    \nabla\Gamma_{3,1/2}(x)
    &=
    -\frac{x}{4\pi|x|^3}
    +\frac{1}{2\pi^2}\frac{x}{|x|^2}
    +O(1).
\end{aligned}
\label{eq:origin-critical-expansion}
\end{equation}

\item[\(n=3\) and \(0<s<\frac12\)]
There exists \(C_{3,s}\in\mathbb R\) such that, as \(|x|\to0\),
\begin{equation}
\begin{aligned}
    \Gamma_{3,s}(x)
    &=
    \frac{1}{4\pi|x|}
    +C_{3,s}
    -\kappa_{3,4-2s}|x|^{1-2s}
    +o\!\left(|x|^{1-2s}\right),\\
    \nabla\Gamma_{3,s}(x)
    &=
    -\frac{x}{4\pi|x|^3}
    -(1-2s)\kappa_{3,4-2s}
      x|x|^{-1-2s}
    +o\!\left(|x|^{-2s}\right).
\end{aligned}
\label{eq:origin-three-subcritical-expansion}
\end{equation}

\item[\(n=2\)]
There exists \(C_{2,s}\in\mathbb R\) such that, as \(|x|\to0\),
\begin{equation}
\begin{aligned}
    \Gamma_{2,s}(x)
    &=
    \frac{1}{2\pi}\log\frac1{|x|}
    +C_{2,s}
    -\kappa_{2,4-2s}|x|^{2-2s}
    +o\!\left(|x|^{2-2s}\right),\\
    \nabla\Gamma_{2,s}(x)
    &=
    -\frac{x}{2\pi|x|^2}
    -(2-2s)\kappa_{2,4-2s}
      x|x|^{-2s}
    +o\!\left(|x|^{1-2s}\right).
\end{aligned}
\label{eq:origin-two-dimensional-expansion}
\end{equation}
\end{description}

\item \textbf{Asymptotics at infinity.}

\begin{description}
\item[\(s\neq\frac12\)]
As \(|x|\to\infty\),
\begin{equation}
\begin{aligned}
    \Gamma_{n,s}(x)
    &=
    \kappa_{n,2s}|x|^{2s-n}
    -\kappa_{n,4s-2}|x|^{4s-2-n}
    +o\!\left(|x|^{4s-2-n}\right),\\
    \nabla\Gamma_{n,s}(x)
    &=
    -(n-2s)\kappa_{n,2s}
      \frac{x}{|x|^{n+2-2s}}
    \\&\qquad
    +(n+2-4s)\kappa_{n,4s-2}
      \frac{x}{|x|^{n+4-4s}}
    +o\!\left(|x|^{4s-3-n}\right).
\end{aligned}
\label{eq:infinity-generic-expansion}
\end{equation}

\item[\(s=\frac12\)]
As \(|x|\to\infty\),
\begin{equation}
\begin{aligned}
    \Gamma_{n,1/2}(x)
    &=
    \kappa_{n,1}|x|^{1-n}
    +\kappa_{n,-1}|x|^{-n-1}
    +o\!\left(|x|^{-n-1}\right),\\
    \nabla\Gamma_{n,1/2}(x)
    &=
    -(n-1)\kappa_{n,1}
      \frac{x}{|x|^{n+1}}
    -(n+1)\kappa_{n,-1}
      \frac{x}{|x|^{n+3}}
    +o\!\left(|x|^{-n-2}\right).
\end{aligned}
\label{eq:infinity-critical-expansion}
\end{equation}
\end{description}
\end{enumerate}
All vector-valued little-\(o\) terms are understood with respect to the
Euclidean norm.
\end{theorem}

Theorem~\ref{thm:fundamental-solution} makes the two-scale structure of
\(\mathscr{L}_s\) explicit. Near the origin, the leading singularity is given by the fundamental
solution of \(-\Delta\), namely the Newtonian kernel when \(n\ge3\)
and the logarithmic kernel when \(n=2\). At infinity, the leading term is the fractional Riesz kernel
\(\kappa_{n,2s}|x|^{2s-n}\). The displayed scale-dependent correction terms quantify the first interaction
between the two parts of the operator.

The gradient expansions are also important for the applications below.
For every \(R>0\),
\[
\Gamma_{n,s}\in W^{1,p}(B_R),
\qquad
\text{for }1\le p<\frac{n}{n-1},
\]
while
\begin{equation}\label{non-int}
\nabla\Gamma_{n,s}\notin
L^{n/(n-1)}(B_R).
\end{equation}
Thus the gradient has precisely the endpoint singularity of the classical
Newtonian kernel. This borderline behavior will be used to eliminate
derivatives of the Dirac mass in the analysis of isolated singularities. 

The classical B\^ocher theorem determines the singular part of a
nonnegative harmonic function at an isolated point. More precisely, if \(u\) is
nonnegative and harmonic in a punctured ball, then
\[
    u=a\Phi+h,
\]
where \(a\geq 0\), \(\Phi\) is the fundamental solution of
\(-\Delta\), and \(h\) is harmonic in the full ball. The result first
appeared in B\^ocher's paper \cite{Bocher1903}. A later proof can
be found in \cite{ABR1992}. Extensions to more general second-order elliptic equations were obtained in several classical works \cite{HW1954,GilbargSerrin1955}. 
The distributional formulation of Brezis and Lions \cite{BL1981} is
particularly relevant to the present work. In their approach, extending
the equation across the puncture produces a defect distribution \(T\)
satisfying $ \operatorname{supp}T\subset\{0\}. $
Related classification results for positive solutions of semilinear
elliptic equations are given in \cite{Gidas1980,GS1981}. For the fractional Laplacian, maximum principles and
B\^ocher-type theorems on punctured balls were established in
\cite{LWX2018}. A distributional theory for nonnegative solutions of
fractional Laplace equations with an isolated singularity was
developed further in \cite{LLWX2020}. More generally, Klimsiak \cite{K2026} recently proved a B\^ocher-type theorem within a broad
framework of elliptic equations with drift-perturbed L\'evy operators,
which includes the mixed operator considered here. His theorem is formulated
for functions that are nonnegative on all of \(\mathbb R^n\), whereas
Theorem~\ref{thm-bocher} below requires nonnegativity only in \(B_1\). His argument is based on the probabilistic potential theory. In the present
paper, we instead give a direct analytic proof based on the precise
asymptotic expansions of \(\Gamma_{n,s}\) and
\(\nabla\Gamma_{n,s}\). 
Once the point-supported defect is shown to have
order at most one, it may still contain the first derivatives of
\(\delta_0\). The Newtonian leading term of \(\nabla\Gamma_{n,s}\) in
Theorem~\ref{thm:fundamental-solution} has exactly the endpoint
non-integrability \eqref{non-int} needed to exclude these dipole terms under the
nonnegativity assumption.

Our second main result is the following distributional B\^ocher theorem.

\begin{theorem}[B\^ocher theorem for $-\Delta+(-\Delta)^s$]
\label{thm-bocher}
Let $B_1\subset\mathbb R^n$, $n\ge2$, and $s\in(0,1)$. Assume
$u\in\mathcal L^{2s}(\mathbb R^n)$ satisfies $u\ge0$ a.e. in $B_1$ and
\[
    \mathscr{L}_su+cu\ge0
    \qquad
    \text{in }\mathcal D'(B_1\setminus\{0\}),
\]
where $c\in L^\infty(B_1)$. Then there exist $a\ge0$ and a nonnegative
Radon measure $\mu$ in $B_1$, with $\mu(\{0\})=0$, such that
\[
    \mathscr{L}_su+cu
    =
    \mu+a\delta_0
    \qquad
    \text{in }\mathcal D'(B_1).
\]
\end{theorem}

In particular, the entire point-supported singular part of a nonnegative
distributional supersolution consists of a nonnegative multiple of
\(\delta_0\). No derivative of the Dirac mass can occur. The theorem is
formulated under the natural global tail condition
\(u\in\mathcal L^{2s}(\mathbb R^n)\), which is needed in order to define
the nonlocal part distributionally.

Maximum principles for local, nonlocal, and mixed elliptic equations have
been studied in many different settings. Representative results can be
found in \cite{JK2006,CLL2017,BDVV2022,SVWZ2025}. The relation between
maximum principles on punctured domains and B\^ocher-type theorems for the
fractional Laplacian is developed in \cite{LWX2018,LLWX2020}. In the
present mixed setting, Theorem~\ref{thm-bocher} provides the mechanism
that allows the puncture to be recovered distributionally. After replacing
the zero-order coefficient by its positive part, the singular contribution
at the puncture has the favorable sign. This reduces the problem to a
quantitative minimum estimate for a global supersolution on the ball.

We obtain the following two principles.

\begin{theorem}[Positive maximum principle]
\label{thm:mixed-positive}
Let $n\ge2$, $s\in(0,1)$, $B_r(x_0)\subset\mathbb R^n$, and
$c\in L^\infty(B_r(x_0))$. Assume
$u\in\mathcal L^{2s}(\mathbb R^n)$,
$u\ge0$ a.e. in $\mathbb R^n$, and
\[
    \mathscr{L}_su+cu\ge0
    \qquad
    \text{in }\mathcal D'(B_r(x_0)\setminus\{x_0\}).
\]
If
\[
    u\ge m>0
    \qquad
    \text{a.e. in }
    B_r(x_0)\setminus B_{r/2}(x_0),
\]
then
\[
    u\ge\alpha m
    \qquad
    \text{a.e. in }B_r(x_0)\setminus\{x_0\},
\]
where
\[
    \alpha
    =
    \alpha\!\left(
        n,s,
        r^{2s}\|c^+\|_{L^\infty(B_r(x_0))}
    \right)>0,
    \qquad
    c^+=\max\{c,0\}.
\]
In particular, if $0<r\le1$, then $\alpha$ may be chosen to depend only
on $n$, $s$, and
$\|c^+\|_{L^\infty(B_r(x_0))}$. If $c\le0$, it may be chosen
independently of both $r$ and $c$.
\end{theorem}

\begin{theorem}[Antisymmetric maximum principle]
\label{thm:mixed-antisymmetric}
Let $n\ge2$, $s\in(0,1)$,
\[
    H=\{x_1>0\},
    \qquad
    \widetilde x=(-x_1,x'),
\]
and let $B_r(x_0)\subset H$ and
$c\in L^\infty(B_r(x_0))$. Assume
$u\in\mathcal L^{2s}(\mathbb R^n)$ is antisymmetric with respect to
$\partial H$, namely
\[
    u(\widetilde x)=-u(x)
    \qquad
    \text{for a.e. }x\in H,
\]
and $u\ge0$ a.e. in $H$. Suppose
\[
    \mathscr{L}_su+cu\ge0
    \qquad
    \text{in }\mathcal D'(B_r(x_0)\setminus\{x_0\}),
\]
and
\[
    u\ge m>0
    \qquad
    \text{a.e. in }
    B_r(x_0)\setminus B_{r/2}(x_0).
\]
Then
\[
    u\ge\alpha m
    \qquad
    \text{a.e. in }B_r(x_0)\setminus\{x_0\},
\]
where
\[
    \alpha
    =
    \alpha\!\left(
        n,s,
        r^{2s}\|c^+\|_{L^\infty(B_r(x_0))}
    \right)>0.
\]
As above, for $0<r\le1$ the constant may be chosen uniformly in $r$,
while for $c\le0$ it depends only on $n$ and $s$.
\end{theorem}

The quantity
\(r^{2s}\|c^+\|_{L^\infty(B_r(x_0))}\)
appearing in the two estimates is natural. At an interior minimum the
local part has the favorable sign. The contribution of the fractional
operator from an annulus whose scale is \(r\) is of order \(r^{-2s}\).
The competition with the zero-order term therefore occurs through the
dimensionless quantity \(r^{2s}\|c^+\|_\infty\). The same scaling enters
the antisymmetric argument after the fractional kernel is paired with its
reflection across the boundary hyperplane.

We briefly explain the main ideas of the proofs. For
Theorem~\ref{thm:fundamental-solution}, we start from the multiplier
\[
m(\xi)
=
\frac{1}{|\xi|^2+|\xi|^{2s}}
\]
and separate low and high frequencies. At high frequencies,
\[
m(\xi)
=
|\xi|^{-2}
-
|\xi|^{-(4-2s)}
+
O\!\left(|\xi|^{-(6-4s)}\right),
\]
while at low frequencies,
\[
m(\xi)
=
|\xi|^{-2s}
-
|\xi|^{2-4s}
+
O\!\left(|\xi|^{4-6s}\right).
\]
The homogeneous terms are inverted using the Fourier transform formulas
for Riesz kernels and their distributional continuation. When a
homogeneous exponent reaches a critical degree, the corresponding
distributional continuation produces the logarithmic contribution.
Cutoff multipliers and dyadic kernel estimates control the remainder
terms together with their derivatives.

The proof of Theorem~\ref{thm-bocher} follows the distributional strategy
of \cite[Theorems~3 and~4]{LWX2018}, adapted to the mixed operator. A
radial barrier first shows that the positive distribution on the punctured
ball has finite mass near the origin. After this measure is extended
across the puncture, the remaining defect is a distribution supported at
the origin. A cutoff argument shows that its order is at most one, and
hence it has the form
$a\delta_0+d\cdot\nabla\delta_0$.
A localized potential representation, together with the critical endpoint
non-integrability of the dipole term furnished by the gradient expansion in
Theorem~\ref{thm:fundamental-solution}, excludes \(d\ne0\). An averaged
estimate near the origin then implies \(a\ge0\).

The two maximum principles are obtained by applying
Theorem~\ref{thm-bocher} to restore the supersolution inequality across
the puncture. After mollification, the resulting inequality can be
evaluated at an interior minimum. In the antisymmetric case, the argument
also uses the strict positivity of the reflected kernel
\[
\frac{1}{|x-y|^{n+2s}}
-
\frac{1}{|x-\widetilde y|^{n+2s}}
>0,
\qquad
x,y\in H.
\]

The paper is organized as follows.
Section~\ref{sec:pf-thm1.1} is devoted to the construction of the
fundamental solution and the proof of its asymptotic expansions in
Theorem~\ref{thm:fundamental-solution}.
Section~\ref{sec:pf-bocher} establishes the B\^ocher-type theorem.
The positive and antisymmetric maximum principles are proved in
Section~\ref{sec:proof-maximum-principles}.
The appendices contain the auxiliary results used in these arguments,
including distributional properties of the mixed operator, Fourier
transform formulas for homogeneous distributions, and the cutoff and
dyadic kernel estimates required for the asymptotic analysis.

\section{Proof of Theorem \ref{thm:fundamental-solution}}\label{sec:pf-thm1.1}

\begin{proof}
\textbf{Step 1. Construction and frequency decomposition.}
Let $m(\xi)=(|\xi|^2+|\xi|^{2s})^{-1}$.
Since $m(\xi)\asymp|\xi|^{-2s}$ as $|\xi|\to0$, $m(\xi)\asymp|\xi|^{-2}$ as $|\xi|\to\infty$,
and \(n>2s\), the function \(m\) belongs to
\(L^1_{\mathrm{loc}}(\mathbb R^n)\) and defines a tempered distribution.
For every \(\varphi\in\mathcal S(\mathbb R^n)\),
\begin{equation*}
    \langle\Gamma_{n,s},\varphi\rangle
    =\frac{1}{(2\pi)^n}
    \int_{\mathbb R^n}m(\xi)\widehat\varphi(-\xi)\,\mathrm{d}\xi,
\end{equation*}
and the integral is absolutely convergent. Moreover,
\begin{equation*}
    |\xi|^2m(\xi),\ |\xi|^{2s}m(\xi)\in L^1_{\mathrm{loc}}(\mathbb R^n)
\end{equation*}
with at most polynomial growth, and
\begin{equation*}
    \bigl(|\xi|^2+|\xi|^{2s}\bigr)m(\xi)=1
\end{equation*}
almost everywhere. Hence \eqref{eq:fundamental-solution-equation} follows
in \(\mathcal S'(\mathbb R^n)\).

Choose a radial cutoff \(\chi\in C_c^\infty(\mathbb R^n)\) satisfying
\begin{equation*}
    \chi=1\quad\text{in }\{|\xi|\leq1\},
    \qquad
    \chi=0\quad\text{in }\{|\xi|\geq2\},
\end{equation*}
and set \(\eta=1-\chi\). Then
\begin{equation}
    \Gamma_{n,s}
    =\mathcal F^{-1}(\chi m)+\mathcal F^{-1}(\eta m).
    \label{eq:frequency-decomposition}
\end{equation}
For every multi-index \(\beta\),
\(\xi^\beta\chi(\xi)m(\xi)\in L^1\). Hence
\(\mathcal F^{-1}(\chi m)\in C^\infty(\mathbb R^n)\). For every multi-index \(\gamma\),
\[
|\partial_\xi^\gamma(\eta m)(\xi)|
\le C_\gamma|\xi|^{-2-|\gamma|},
\quad\text{for any } |\xi|\ge1.
\]
Lemma~\ref{lem:cutoff-symbol-estimates} therefore shows that
\(\mathcal F^{-1}(\eta m)\) is smooth away from the origin and rapidly
decreasing at infinity. Since \(m\), \(\chi\), and \(\eta\) are real-valued
and radial, both terms in \eqref{eq:frequency-decomposition} are
real-valued and radial.

\medskip
\textbf{Step 2. The high-frequency expansion.}
The exact identity
\begin{equation}
    m(\xi)
    =|\xi|^{-2}-|\xi|^{-(4-2s)}+R_\infty(\xi),
    \qquad
    R_\infty(\xi)
    =\frac{|\xi|^{-(6-4s)}}{1+|\xi|^{2s-2}},
    \label{eq:high-frequency-expansion}
\end{equation}
follows from \((1+z)^{-1}=1-z+z^2(1+z)^{-1}\), with
\(z=|\xi|^{2s-2}\). For every multi-index \(\rho\),
\begin{equation}
    |\partial_\xi^\rho R_\infty(\xi)|
    \leq C_\rho|\xi|^{-(6-4s)-|\rho|}
    \qquad\text{for any } |\xi|\geq1.
    \label{eq:high-frequency-remainder-estimate}
\end{equation}
Indeed, \(|\xi|^{2s-2}\leq1\) in this region, and differentiation of
radial powers gives
\(|\partial_\xi^\rho|\xi|^a|\leq C_{a,\rho}|\xi|^{a-|\rho|}\).
The Leibniz and chain rules then give
\eqref{eq:high-frequency-remainder-estimate}.

\medskip
\textbf{Step 3. The case \(n\geq3\) and \(4-2s<n\).}
Corollary~\ref{cor:cutoff-homogeneous-terms} yields
\begin{align}
    \mathcal F^{-1}(\eta|\xi|^{-2})(x)
    &=\frac{|x|^{2-n}}{(n-2)\omega_{n-1}}+H_0(x),
    \label{eq:origin-first-term}
    \\
    \mathcal F^{-1}(\eta|\xi|^{-(4-2s)})(x)
    &=\kappa_{n,4-2s}|x|^{4-2s-n}+H_1(x),
    \label{eq:origin-second-term}
\end{align}
where \(H_0,H_1\) are smooth and radial near the origin. The remaining
smooth contribution \(\mathcal F^{-1}(\chi m)\) has the same property.
Thus these regular terms are bounded, and their gradients are bounded.

Set \(\mu=6-4s\). By
\eqref{eq:high-frequency-remainder-estimate} and
Lemma~\ref{lem:cutoff-symbol-estimates},
\begin{equation*}
    \mathcal F^{-1}(\eta R_\infty)(x)
    =
    \begin{cases}
        O(|x|^{\mu-n}),&\mu<n,\\
        O(1+|\log|x||),&\mu=n,\\
        O(1),&\mu>n,
    \end{cases}
    \qquad |x|\to0.
\end{equation*}
If \(\mu<n\), then
\begin{equation*}
    \frac{|x|^{\mu-n}}{|x|^{4-2s-n}}
    =|x|^{2-2s}\to0.
\end{equation*}
If \(\mu\geq n\), then by \(4-2s-n<0\), so logarithmic and bounded
terms are also \(o(|x|^{4-2s-n})\). Hence
\begin{equation}
    \mathcal F^{-1}(\eta R_\infty)(x)
    =o\!\left(|x|^{4-2s-n}\right).
    \label{eq:origin-generic-function-remainder}
\end{equation}

Applying the differentiated estimate in
Lemma~\ref{lem:cutoff-symbol-estimates} gives
\begin{equation*}
    \nabla\mathcal F^{-1}(\eta R_\infty)(x)
    =
    \begin{cases}
        O(|x|^{\mu-n-1}),&\mu<n+1,\\
        O(1+|\log|x||),&\mu=n+1,\\
        O(1),&\mu>n+1.
    \end{cases}
\end{equation*}
In the first case the quotient by \(|x|^{3-2s-n}\) is again
\(|x|^{2-2s}\), in the remaining cases we use \(3-2s-n<0\). Therefore
\begin{equation}
    \nabla\mathcal F^{-1}(\eta R_\infty)(x)
    =o\!\left(|x|^{3-2s-n}\right).
    \label{eq:origin-generic-gradient-remainder}
\end{equation}
Combining \eqref{eq:frequency-decomposition},
\eqref{eq:high-frequency-expansion}, \eqref{eq:origin-first-term},
\eqref{eq:origin-second-term},
\eqref{eq:origin-generic-function-remainder}, and
\eqref{eq:origin-generic-gradient-remainder} proves
\eqref{eq:origin-generic-expansion}. The derivative of the second
homogeneous term is
\begin{equation*}
    \nabla\!\left(-\kappa_{n,4-2s}|x|^{4-2s-n}\right)
    =-(4-2s-n)\kappa_{n,4-2s}x|x|^{2-2s-n}.
\end{equation*}

\medskip
\textbf{Step 4. The case \(n=3\) and \(s=\frac12\).}
Retaining one additional term gives the exact decomposition
\begin{equation}
    m(\xi)
    =|\xi|^{-2}-|\xi|^{-3}+|\xi|^{-4}
     -\frac{|\xi|^{-5}}{1+|\xi|^{-1}}.
    \label{eq:critical-high-frequency-expansion}
\end{equation}
The first term contributes \((4\pi|x|)^{-1}\), modulo a smooth radial
function. Since \(\eta\) vanishes near the origin,
\(\eta|\xi|^{-3}=\eta\operatorname{Fp}|\xi|^{-3}\), and
\eqref{eq:finite-part-fourier-transform} gives
\begin{equation*}
    \mathcal F^{-1}(\eta|\xi|^{-3})(x)
    =\frac{1}{2\pi^2}\log\frac1{|x|}+H_2(x),
\end{equation*}
where \(H_2\) is smooth and radial near the origin. Moreover,
\begin{equation*}
    \mathcal F^{-1}(\eta|\xi|^{-4})(x)
    =\kappa_{3,4}|x|+H_3(x),
    \qquad \kappa_{3,4}=-\frac{1}{8\pi},
\end{equation*}
so this term and its gradient are bounded. The last multiplier in
\eqref{eq:critical-high-frequency-expansion} is a symbol of order
\(-5\), since \(5>3+1\), its inverse Fourier transform and its gradient
are bounded near the origin. The signs in
\eqref{eq:critical-high-frequency-expansion} prove
\eqref{eq:origin-critical-expansion}.

\medskip
\textbf{Step 5. The case \(n=3\) and \(0<s<\frac12\).}
Here \(4-2s\in(3,4)\), so this exponent is nonexceptional. Hence
\begin{equation}
    \mathcal F^{-1}(\eta|\xi|^{-(4-2s)})(x)
    =\kappa_{3,4-2s}|x|^{1-2s}+H_4(x),
    \label{eq:three-dimensional-subcritical-term}
\end{equation}
where \(H_4\) is smooth and radial. The remainder \(\eta R_\infty\)
is radial, has order \(-(6-4s)\), and
\begin{equation*}
    0<1-2s<\min\{(6-4s)-3,2\}.
\end{equation*}
The radial regularity statement in
Lemma~\ref{lem:cutoff-symbol-estimates}, with \(\beta=1-2s\), gives
\begin{equation*}
    \mathcal F^{-1}(\eta R_\infty)(x)
    =\mathcal F^{-1}(\eta R_\infty)(0)+o(|x|^{1-2s}),
    \qquad
    \nabla\mathcal F^{-1}(\eta R_\infty)(x)=o(|x|^{-2s}).
\end{equation*}
Every smooth radial function \(H\) satisfies
\begin{equation*}
    H(x)=H(0)+O(|x|^2),
    \qquad
    \nabla H(x)=O(|x|).
\end{equation*}
Collecting all constant terms into \(C_{3,s}\) and using
\eqref{eq:high-frequency-expansion} and
\eqref{eq:three-dimensional-subcritical-term} proves
\eqref{eq:origin-three-subcritical-expansion}.

\medskip
\textbf{Step 6. The case \(n=2\).}
The leading term is critical. By \eqref{eq:high-frequency-cutoff-critical}, we have
\begin{equation*}
    \mathcal F^{-1}(\eta|\xi|^{-2})(x)
    =\frac{1}{2\pi}\log\frac1{|x|}+H_5(x),
\end{equation*}
with \(H_5\) smooth and radial. Since \(4-2s\in(2,4)\) is
nonexceptional,
\begin{equation*}
    \mathcal F^{-1}(\eta|\xi|^{-(4-2s)})(x)
    =\kappa_{2,4-2s}|x|^{2-2s}+H_6(x).
\end{equation*}
Furthermore,
\begin{equation*}
    0<2-2s<\min\{(6-4s)-2,2\}.
\end{equation*}
Applying the radial regularity statement in
Lemma~\ref{lem:cutoff-symbol-estimates} with \(\beta=2-2s\), and using
\(H(x)=H(0)+O(|x|^2)\), \(\nabla H(x)=O(|x|)\) for smooth radial
functions, gives \eqref{eq:origin-two-dimensional-expansion} after the
constant terms are collected into \(C_{2,s}\).

\medskip
\textbf{Step 7. The case \(|x|\to\infty\) and \(s\neq\frac12\).}
The exact low-frequency expansion is
\begin{equation}
    m(\xi)
    =|\xi|^{-2s}-|\xi|^{2-4s}+R_0(\xi),
    \qquad
    R_0(\xi)=\frac{|\xi|^{4-6s}}{1+|\xi|^{2-2s}}.
    \label{eq:low-frequency-expansion}
\end{equation}
For every multi-index \(\rho\),
\begin{equation*}
    |\partial_\xi^\rho R_0(\xi)|
    \leq C_\rho|\xi|^{4-6s-|\rho|}
    \qquad\text{for any }0<|\xi|\leq2.
\end{equation*}
Since \(4s-2\in(-2,2)\setminus\{0\}\),
Corollary~\ref{cor:cutoff-homogeneous-terms} yields, for every
\(N>0\),
\begin{align*}
    \mathcal F^{-1}(\chi|\xi|^{-2s})(x)
    &=\kappa_{n,2s}|x|^{2s-n}+O(|x|^{-N}),\\
    \mathcal F^{-1}(\chi|\xi|^{2-4s})(x)
    &=\kappa_{n,4s-2}|x|^{4s-2-n}+O(|x|^{-N}).
\end{align*}
Because \(4-6s>-n\), the low-frequency estimate in
Lemma~\ref{lem:cutoff-symbol-estimates} gives
\begin{equation*}
    \mathcal F^{-1}(\chi R_0)(x)=O(|x|^{6s-4-n}),
    \qquad
    \nabla\mathcal F^{-1}(\chi R_0)(x)=O(|x|^{6s-5-n}).
\end{equation*}
The ratios to the respective second terms are
\begin{equation*}
    \frac{|x|^{6s-4-n}}{|x|^{4s-2-n}}
    =\frac{|x|^{6s-5-n}}{|x|^{4s-3-n}}
    =|x|^{2s-2}\to0.
\end{equation*}
The high-frequency contribution \(\mathcal F^{-1}(\eta m)\) is rapidly
decreasing. Combining these estimates with
\eqref{eq:low-frequency-expansion} proves
\eqref{eq:infinity-generic-expansion}.

\medskip
\textbf{Step 8. The case \(|x|\to\infty\) and \(s=\frac12\).}
The finite identity
\begin{equation}
    m(\xi)
    =|\xi|^{-1}-1+|\xi|-\frac{|\xi|^2}{1+|\xi|}
    \label{eq:critical-low-frequency-expansion}
\end{equation}
replaces \eqref{eq:low-frequency-expansion}. For every \(N>0\),
\begin{equation*}
    \mathcal F^{-1}(\chi|\xi|^{-1})(x)
    =\kappa_{n,1}|x|^{1-n}+O(|x|^{-N}),
\end{equation*}
and
\begin{equation*}
    \mathcal F^{-1}(\chi|\xi|)(x)
    =\kappa_{n,-1}|x|^{-n-1}+O(|x|^{-N}).
\end{equation*}
The inverse Fourier transform of \(\chi\) is a Schwartz function. The
last multiplier in \eqref{eq:critical-low-frequency-expansion} satisfies
the low-frequency symbol estimates with exponent \(2\). Hence
\begin{equation*}
    \mathcal F^{-1}\!\left(\chi\frac{|\xi|^2}{1+|\xi|}\right)(x)
    =O(|x|^{-n-2}),
    \qquad
    \nabla\mathcal F^{-1}\!\left(\chi\frac{|\xi|^2}{1+|\xi|}\right)(x)
    =O(|x|^{-n-3}).
\end{equation*}
Together with the rapid decay of \(\mathcal F^{-1}(\eta m)\), this proves
\eqref{eq:infinity-critical-expansion}.  Notice that every gradient
remainder used above comes directly from the differentiated symbol
estimates in Lemma~\ref{lem:cutoff-symbol-estimates}. In particular, no
undifferentiated little-$o$ term is differentiated.
\end{proof}

\section{Proof of Theorem \ref{thm-bocher}}\label{sec:pf-bocher}

\begin{proof}
Put $\mathscr{L}_s=-\Delta+(-\Delta)^s$ and $T=\mathscr{L}_su+cu\in\mathcal D'(B_1\setminus\{0\})$.
Since $T\ge0$, it is a nonnegative Radon measure on
$B_1\setminus\{0\}$, which we denote by $\mu$. We extend $\mu$ to $B_1$ by $\mu(\{0\})=0$. The argument below follows the
distributional scheme of \cite[Theorems~3 and~4]{LWX2018}, with the
mixed local--nonlocal estimates made explicit.

\medskip
\noindent\textbf{Step 1. $\mu$ has finite mass near $0$.}
Set $M=\|c\|_{L^\infty(B_1)}$.
Choose a radial $\eta\in C_c^\infty(B_2)$ with $\eta=1$ in $B_1$, and fix $0<\beta<2s$ and $0<\alpha<1$.
For $0<\varepsilon<e^{-1}$ define
\begin{equation}\label{eq:bocher-barrier}
\phi_\varepsilon(r)=
\begin{cases}
1+\eta(2r)\displaystyle\left[r^\beta-
\left(\dfrac{\log r}{\log\varepsilon}\right)^\alpha\right],&n=2,\\[1.2ex]
1+\eta(r)r^\beta-\left(\dfrac\varepsilon r\right)^\alpha,&n\ge3.
\end{cases}
\end{equation}
For $r<1/4$ the cutoffs in \eqref{eq:bocher-barrier} are equal to one.
If $n\ge3$, then
\begin{align}
-\Delta\phi_\varepsilon(r)
&=-\beta(\beta+n-2)r^{\beta-2}
  -\alpha(n-2-\alpha)\varepsilon^\alpha r^{-\alpha-2},
\notag\\
|(-\Delta)^s\phi_\varepsilon(r)|
&\le C r^{\beta-2s}
   +C\varepsilon^\alpha r^{-\alpha-2s}.
\label{eq:bocher-frac-nge3}
\end{align}
Indeed, since \(0<\alpha<1<n-2s\),
Proposition~\ref{prop:homogeneous-fourier-transform} gives
\[
(-\Delta)^s(|x|^{-\alpha})
=C_{n,s,\alpha}|x|^{-\alpha-2s}
\qquad\text{in }\mathbb R^n\setminus\{0\}.
\]
It remains to estimate \(F(x)=\eta(|x|)|x|^\beta\). For
\(r=|x|<1/4\), the symmetric integral formula gives
\[
(-\Delta)^sF(x)
=\frac{C_{n,s}}2\int_{\mathbb R^n}
\frac{2F(x)-F(x+z)-F(x-z)}{|z|^{n+2s}}\,\mathrm{d}z.
\]
Since \(F\) is smooth in \(B_{r/2}(x)\),
\[
|2F(x)-F(x+z)-F(x-z)|
\le
\begin{cases}
C r^{\beta-2}|z|^2,& |z|<r/2,\\
C|z|^\beta,& r/2\le |z|\le3,\\
C,& |z|>3.
\end{cases}
\]
Because \(0<\beta<2s<2\), integration in the three regions yields
\(|(-\Delta)^sF(x)|\le Cr^{\beta-2s}\). Together with the estimate for
\(|x|^{-\alpha}\), this proves \eqref{eq:bocher-frac-nge3}, with \(C\)
independent of \(\varepsilon\).

If $n=2$, let $L=\log(1/r)$ and $E=\log(1/\varepsilon)$.
Then
\begin{align*}
-\Delta\phi_\varepsilon(r)
&=-\beta^2r^{\beta-2}
  -\alpha(1-\alpha)E^{-\alpha}L^{\alpha-2}r^{-2},
\\
|(-\Delta)^s\phi_\varepsilon(r)|
&\le Cr^{\beta-2s}
 +CE^{-\alpha}r^{-2s}(1+L^\alpha).
\end{align*}
Indeed, for
$F(x)=\eta(2|x|)(\log(1/|x|))^\alpha$, splitting the defining integral
into $B_{r/2}(x)$ and its complement gives
\[
|(-\Delta)^sF(x)|\le Cr^{-2s}(1+L^\alpha),
\qquad r=|x|<\frac14.
\]
Since
\[
r^{2-2s}\to0,
\qquad
r^{2-2s}(L^{2-\alpha}+L^2)\to0,
\quad\text{as }r\downarrow0,
\]
the preceding estimates imply that there exists
$r_0\in(0,1/8)$, independent of $\varepsilon$, such that
\[
(\mathscr{L}_s+M)\phi_\varepsilon\le0
\quad\text{in }\mathcal D'(B_{2r_0}).
\]
Set $\psi_\varepsilon=(\phi_\varepsilon)^+$. Proposition~\ref{construct-super}
gives
\[
(\mathscr{L}_s+M)\psi_\varepsilon\le0
\quad\text{in }\mathcal D'(B_{2r_0}).
\]
Setting $\psi_\varepsilon^\delta=\psi_\varepsilon*j_\delta$,
Lemma~\ref{mollifi} gives, for all sufficiently small $\delta>0$,
\[
(\mathscr{L}_s+M)\psi_\varepsilon^\delta\le0
\quad\text{in }B_{r_0}.
\]
Choose $h\in C_c^\infty(B_{3r_0/4})$ with $0\le h\le1$ and $h=1$ in $B_{r_0/2}$.
Taking $\delta$ smaller if necessary,
$h\psi_\varepsilon^\delta\in C_c^\infty(B_1\setminus\{0\})$.  Put
$\psi=\psi_\varepsilon^\delta$. By the product identity
\begin{equation*}
\begin{aligned}
\mathscr{L}_s(h\psi)(x)
&=h(x)\mathscr{L}_s\psi(x)+\psi(x)\mathscr{L}_sh(x)
  -2\nabla h(x)\cdot\nabla\psi(x)\\
&\quad-C_{n,s}\int_{\mathbb R^n}
\frac{(h(x)-h(y))(\psi(x)-\psi(y))}{|x-y|^{n+2s}}\,\mathrm{d}y,
\end{aligned}
\end{equation*}
we have
\begin{equation*}
\begin{aligned}
0&\le\int_{B_1}h\psi\,\mathrm{d}\mu
=T(h\psi)\\
&=\int_{B_1}u\,\mathscr{L}_s(h\psi)\,\mathrm{d}x
  +\int_{B_1}cuh\psi\,\mathrm{d}x
  +\int_{\mathbb R^n\setminus B_1}u\,\mathscr{L}_s(h\psi)\,\mathrm{d}x\\
&\le \int_{B_1}u\psi\,\mathscr{L}_sh\,\mathrm{d}x
 -2\int_{B_1}u\,\nabla h\cdot\nabla\psi\,\mathrm{d}x\\
&\quad-C_{n,s}\int_{B_1}u(x)
 \int_{\mathbb R^n}
 \frac{(h(x)-h(y))(\psi(x)-\psi(y))}{|x-y|^{n+2s}}\,\mathrm{d}y\,\mathrm{d}x
 +\int_{\mathbb R^n\setminus B_1}u\,\mathscr{L}_s(h\psi)\,\mathrm{d}x.
\end{aligned}
\end{equation*}
where the last inequality uses the fact that
\[
\int_{B_1}uh\bigl[(\mathscr{L}_s+M)\psi\bigr] \,\mathrm{d}x
+\int_{B_1}(c-M)uh\psi\,\mathrm{d}x\le0.
\]
Uniformly in $\varepsilon$ and $\delta$,
\begin{align}
&0\le\psi\le C,
\qquad
|\nabla\psi|\le C
\quad\text{on }\operatorname{supp}\nabla h,
\notag\\&
\sup_{x\in B_1}
\left|
\int_{\mathbb R^n}
\frac{(h(x)-h(y))(\psi(x)-\psi(y))}{|x-y|^{n+2s}}\,\mathrm{d}y
\right|\le C,
\label{eq:commutator-uniform}\\&
|\mathscr{L}_s(h\psi)(x)|\le
\frac{C}{1+|x|^{n+2s}},
\qquad x\in\mathbb R^n\setminus B_1.
\notag
\end{align}
Indeed, in \eqref{eq:commutator-uniform} either $h$ is constant near $x$,
or both differences are $O(|x-y|)$. Hence, the local singularity is
$|x-y|^{-n-2s+2}\in L^1_{\rm loc}$ because $s<1$. To make the
bound uniform, let \(K=\operatorname{supp}\nabla h\). Since
\(K\Subset B_{r_0}\setminus\{0\}\), choose \(\rho>0\) so that the
\(3\rho\)-neighborhood of \(K\) is contained in
\(B_{r_0}\setminus\{0\}\). The explicit barriers and mollification give
\[
\|\nabla\psi_\varepsilon^\delta\|_{L^\infty(
\{\operatorname{dist}(x,K)<3\rho\})}\le C.
\]
If \(\operatorname{dist}(x,K)\ge2\rho\), then \(h\) is constant on
\(B_\rho(x)\). Otherwise, when \(|x-y|<\rho\), both factors in the
numerator are bounded by \(C|x-y|\). Hence, the local part is bounded by
\(C\int_{|z|<\rho}|z|^{-n-2s+2}\,\mathrm{d}z\). In
\(|x-y|\ge\rho\), the uniform bounds of \(h\) and
\(\psi_\varepsilon^\delta\) give the bound
\(C\int_{|z|\ge\rho}|z|^{-n-2s}\,\mathrm{d}z\). This proves
\eqref{eq:commutator-uniform} with a constant independent of
\(\varepsilon\) and \(\delta\).  Consequently,
\begin{equation}\label{eq:mu-finite-bound}
0\le\int_{B_{r_0/2}}\psi_\varepsilon^\delta\,\mathrm{d}\mu
\le C\left(
\int_{B_1}u\,\mathrm{d}x
+\int_{\mathbb R^n\setminus B_1}
\frac{|u(x)|}{1+|x|^{n+2s}}\,\mathrm{d}x
\right)\le C.
\end{equation}
For fixed $\varepsilon$, $\psi_\varepsilon^\delta\to\psi_\varepsilon$ locally uniformly as $\delta\downarrow0$, so Fatou's lemma and \eqref{eq:mu-finite-bound} give $\int_{B_{r_0/2}}\psi_\varepsilon\,\mathrm{d}\mu\le C$. On $0<|x|<r_0/2$ we have $\psi_\varepsilon(x)\to1+|x|^\beta$ as $\varepsilon\downarrow0$; a second application of Fatou's lemma yields $\mu(B_{r_0/2})<\infty$. Hence $\mu$ is a Radon measure on $B_1$.

\medskip
\noindent\textbf{Step 2. The singular distribution has order at most one.}
Set $S=\mathscr{L}_su+cu-\mu\in\mathcal D'(B_1)$, so $\operatorname{supp}S\subset\{0\}$.
Let $H\in C_c^\infty(B_1)$ satisfy $H(0)=0$ and $\nabla H(0)=0$.
Choose $\rho\in C_c^\infty(B_2)$ with $\rho=1$ in $B_1$ and put $H_\varepsilon(x)=\rho(x/\varepsilon)H(x)$.
Since $S(H)=S(H_\varepsilon)$,
\[
|H_\varepsilon|\le C\varepsilon^2,
\qquad
\operatorname{supp}H_\varepsilon\subset B_{2\varepsilon},
\qquad
\|D^2H_\varepsilon\|_{L^\infty}\le C.
\]
Writing $H_\varepsilon(x)=\varepsilon^2G_\varepsilon(x/\varepsilon)$,
with $\|G_\varepsilon\|_{C^2}\le C$, gives
\begin{equation*}
|(-\Delta)^sH_\varepsilon(x)|\le
\begin{cases}
C\varepsilon^{2-2s},&|x|<4\varepsilon,\\
C\varepsilon^{n+2}|x|^{-n-2s},&|x|\ge4\varepsilon.
\end{cases}
\end{equation*}
Therefore
\begin{align*}
|S(H)|=|S(H_\varepsilon)|
&\le C\int_{B_{2\varepsilon}}|u|\,\mathrm{d}x
 +C\varepsilon^{2-2s}\int_{B_{4\varepsilon}}|u|\,\mathrm{d}x\\
&\quad+C\varepsilon^{n+2}
\int_{\mathbb R^n\setminus B_{4\varepsilon}}
\frac{|u(x)|}{|x|^{n+2s}}\,\mathrm{d}x
 +C\varepsilon^2\int_{B_{2\varepsilon}}|cu|\,\mathrm{d}x
 +C\varepsilon^2\mu(B_{2\varepsilon})\to0.
\end{align*}
Here, for the third term, we split $4\varepsilon<|x|<1$ and $|x|\ge1$ and use
\[
\varepsilon^{n+2}
\int_{4\varepsilon<|x|<1}
\frac{|u(x)|}{|x|^{n+2s}}\,\mathrm{d}x
\le C\varepsilon^{2-2s}\int_{B_1}|u|\,\mathrm{d}x.
\]
Hence
\[
S(H)=0
\quad\text{whenever}\quad
H(0)=|\nabla H(0)|=0.
\]
Let $\zeta\in C_c^\infty(B_1)$ satisfy $\zeta=1$ near $0$, and define
\begin{align*}
a&=S(\zeta)
=\int_{\mathbb R^n}u\,\mathscr{L}_s\zeta\,\mathrm{d}x
 +\int_{B_1}cu\zeta\,\mathrm{d}x-
 \int_{B_1}\zeta\,\mathrm{d}\mu,
\\
d_j&=-S(x_j\zeta)
=-\int_{\mathbb R^n}u\,\mathscr{L}_s(x_j\zeta)\,\mathrm{d}x
 -\int_{B_1}cu\,x_j\zeta\,\mathrm{d}x
 +\int_{B_1}x_j\zeta\,\mathrm{d}\mu.
\end{align*}
For any $\varphi\in C_c^\infty(B_1)$, set $H_\varphi=\varphi-(\varphi(0)+\nabla\varphi(0)\cdot x)\zeta$. It
satisfies $H_\varphi(0)=|\nabla H_\varphi(0)|=0$.  Thus
\[
S(\varphi)=a\varphi(0)-d\cdot\nabla\varphi(0),
\]
and therefore
\begin{equation}\label{eq:B1-equation-a-d-new}
\mathscr{L}_su+cu=\mu+a\delta_0+d\cdot\nabla\delta_0
\quad\text{in }\mathcal D'(B_1).
\end{equation}

\medskip
\noindent\textbf{Step 3. We show that $d = 0$.}
Let $\Gamma=\Gamma_{n,s}$ be the fundamental solution of $\mathscr L_s$. Set $w=u-a\Gamma-d\cdot\nabla\Gamma$ and write $\mathrm{d}\nu=\mathrm{d}\mu-cu\,\mathrm{d}x$.
Then
\[
\mathscr{L}_sw=\nu\quad\text{in }\mathcal D'(B_1).
\]
Choose $\xi\in C_c^\infty(B_{7/8})$, $\xi=1$ in $B_{3/4}$, and define
\begin{equation}\label{eq:v-def-new}
v(x)=\int_{B_{7/8}}\Gamma(x-y)\xi(y)\,\mathrm{d}\nu(y).
\end{equation}
By Theorem~\ref{thm:fundamental-solution}, $\Gamma,\nabla\Gamma\in L^p_{\rm loc}(\mathbb R^n)$ and $v\in W^{1,p}_{\rm loc}(\mathbb R^n)$ for every $1<p<n/(n-1)$.
Moreover
\[
\int_{\mathbb R^n}\frac{|v(x)|}{1+|x|^{n+2s}}\,\mathrm{d}x
\le |\nu|(B_{7/8})
\sup_{|y|<7/8}
\int_{\mathbb R^n}
\frac{|\Gamma(x-y)|}{1+|x|^{n+2s}}\,\mathrm{d}x<\infty.
\]
The asymptotics in Theorem~\ref{thm:fundamental-solution} also give $\Gamma,\nabla\Gamma\in\mathcal L^{2s}(\mathbb R^n)$; hence $w-v\in\mathcal L^{2s}(\mathbb R^n)$ and satisfies
\[
\mathscr{L}_s(w-v)=0\quad\text{in }\mathcal D'(B_{3/4}).
\]
Theorem~\ref{smooth-harmonic} gives $w-v\in C^\infty(B_{3/4})$ and hence $w\in W^{1,p}(B_{1/2})$ for every $1<p<\frac n{n-1}$. Choosing any such $p>1$, its Sobolev conjugate satisfies $p^*=np/(n-p)>n/(n-1)$. Thus, Sobolev embedding, together with $\Gamma\in L^{n/(n-1)}(B_{1/2})$, gives
$w+a\Gamma\in L^{n/(n-1)}(B_{1/2})$.
If $d\ne0$, Theorem~\ref{thm:fundamental-solution} gives, as \(|x|\to0\),
\[
d\cdot\nabla\Gamma(x)
=-\frac{d\cdot x}{\omega_{n-1}|x|^n}+o(|x|^{1-n}).
\]
Set
\[
\mathcal C_d
=\left\{x\ne0:d\cdot x\ge\frac12|d|\,|x|\right\},
\]
then there exists
\(r_d>0\) such that
\[
(d\cdot\nabla\Gamma)^-(x)
\ge\frac{|d|}{4\omega_{n-1}}|x|^{1-n},
\qquad x\in\mathcal C_d,\quad0<|x|<r_d.
\]
and hence
\[
(d\cdot\nabla\Gamma)^-\notin L^{n/(n-1)}(B_{1/2}).
\]
On the other hand,
\[
d\cdot\nabla\Gamma=u-(w+a\Gamma),
\qquad
(d\cdot\nabla\Gamma)^-
\le u^-+|w+a\Gamma|=|w+a\Gamma|\in L^{n/(n-1)}(B_{1/2}),
\]
which yields a contradiction. Thus $d=0$.

\medskip
\noindent\textbf{Step 4. Finally, we prove that $a\ge0$.}
Now
\[
\mathscr{L}_su=\nu+a\delta_0,
\qquad
\nu=\mu-cu\,\mathrm{d}x,
\qquad
|\nu|(\{0\})=0.
\]
With $v$ as in \eqref{eq:v-def-new},
\begin{equation}\label{eq:u-decomp-new}
u=v+a\Gamma+H
\quad\text{in }B_{1/2},
\qquad
H\in C^\infty(B_{1/2}).
\end{equation}

Assume first $n\ge3$. By Theorem \ref{thm:fundamental-solution}, we have $|\Gamma(z)|\le C|z|^{2-n}$ in $B_1$. Hence, for $0<\delta<\varepsilon/4<1/8$, \eqref{eq:v-def-new} gives
\begin{equation*}
\begin{aligned}
\frac1{\delta^2}\int_{B_\delta}|v(x)|\,\mathrm{d}x
&\le\frac C{\delta^2}
\int_{B_{2\varepsilon}}
\int_{B_\delta}|x-y|^{2-n}\,\mathrm{d}x\,\mathrm{d}|\nu|(y)\\
&\quad+\frac C{\delta^2}
\int_{B_{7/8}\setminus B_{2\varepsilon}}
\int_{B_\delta}|x-y|^{2-n}\,\mathrm{d}x\,\mathrm{d}|\nu|(y)\\
&\le C|\nu|(B_{2\varepsilon})
 +C\left(\frac\delta\varepsilon\right)^{n-2}|\nu|(B_{7/8}),
\end{aligned}
\end{equation*}
where we use
\[
\sup_{y\in\mathbb R^n}
\int_{B_\delta}|x-y|^{2-n}\,\mathrm{d}x\le C\delta^2.
\]
Fixing $\varepsilon$ and letting $\delta\downarrow0$ in the preceding bound gives
$\limsup_{\delta\downarrow0}\delta^{-2}\int_{B_\delta}|v|\,\mathrm{d}x\le C|\nu|(B_{2\varepsilon})$. Since $|\nu|(\{0\})=0$, letting $\varepsilon\downarrow0$ yields
\[
\frac1{\delta^2}\int_{B_\delta}|v|\,\mathrm{d}x\to0
\qquad\text{as }\delta\downarrow0.
\]
Theorem~\ref{thm:fundamental-solution} and $H\in C^\infty(B_{1/2})$ also give
\begin{equation}\label{eq:Gamma-average-nge3}
\frac1{\delta^2}\int_{B_\delta}\Gamma(x)\,\mathrm{d}x
\to\frac1{2(n-2)},
\qquad
\frac1{\delta^2}\int_{B_\delta}|H|\,\mathrm{d}x\to0,\quad\text{as }\delta\rightarrow0.
\end{equation}
If $a<0$, \eqref{eq:u-decomp-new}--\eqref{eq:Gamma-average-nge3} imply
\[
\limsup_{\delta\downarrow0}
\frac1{\delta^2}\int_{B_\delta}u\,\mathrm{d}x
\le\frac{a}{2(n-2)}<0,
\]
contrary to $u\ge0$.

For $n=2$, by Theorem \ref{thm:fundamental-solution}, we have  
$
|\Gamma(z)|\le C(\log\frac1{|z|}+1)
$
in $B_{1}$. Hence, for $0<\delta<\varepsilon/4$,
\begin{align*}
\frac1{\delta^2\log(1/\delta)}
\int_{B_\delta}|v(x)|\,\mathrm{d}x
&\le C|\nu|(B_{2\varepsilon})+C\frac{1+\log(1/\varepsilon)}{\log(1/\delta)}
|\nu|(B_{7/8}).
\end{align*}
Fixing $\varepsilon$ and letting $\delta\downarrow0$ first gives
$\limsup_{\delta\downarrow0}[\delta^2\log(1/\delta)]^{-1}\int_{B_\delta}|v|\,\mathrm{d}x\le C|\nu|(B_{2\varepsilon})$. Since $|\nu|(\{0\})=0$, letting $\varepsilon\downarrow0$ yields
\[
\frac1{\delta^2\log(1/\delta)}\int_{B_\delta}|v|\,\mathrm{d}x\to0
\qquad\text{as }\delta\downarrow0.
\]
Moreover,
\begin{equation*}
\frac1{\delta^2\log(1/\delta)}
\int_{B_\delta}\Gamma(x)\,\mathrm{d}x\to\frac12,
\qquad
\frac1{\delta^2\log(1/\delta)}
\int_{B_\delta}|H|\,\mathrm{d}x\to0,\quad\text{as }\delta\rightarrow0.
\end{equation*}
Again, $a<0$ contradicts $u\ge0$.  Therefore, $a\ge0$, and
\eqref{eq:B1-equation-a-d-new} becomes
\[
\mathscr{L}_su+cu=\mu+a\delta_0
\quad\text{in }\mathcal D'(B_1).
\]
\end{proof}

\section{Proofs of the maximum principles}
\label{sec:proof-maximum-principles}

The following argument is the mixed-operator
analogue of \cite[Theorems~1 and~2]{LWX2018}.

\begin{proof}[Proof of Theorem~\ref{thm:mixed-positive}]
Set $\Lambda=\|c^+\|_{L^\infty(B_r(x_0))}$.
Since $\Lambda-c\ge0$ and $u\ge0$ in $B_r(x_0)$,
\[
\mathscr{L}_su+\Lambda u=(\mathscr{L}_su+cu)+(\Lambda-c)u\ge0
\quad\text{in }\mathcal D'(B_r(x_0)\setminus\{x_0\}).
\]
The corresponding ball-version of Theorem~\ref{thm-bocher}, applied with the constant coefficient $\Lambda$, then gives
\[
\mathscr{L}_su+\Lambda u\ge0\quad\text{in }\mathcal D'(B_r(x_0)).
\]
We first assume that $u$ is smooth in the ball.  Let $x_*\in\overline
{B_{r/2}(x_0)}$ be a minimum point.  If $u(x_*)\ge m$ there is nothing
to prove.  Otherwise $x_*$ is an interior point and
$-\Delta u(x_*)\le0$.  Since $u(x_*)\le u$ in $B_{r/2}(x_0)$,
$u\ge m$ on $B_r(x_0)\setminus B_{r/2}(x_0)$, and $u\ge0$ in
$\mathbb R^n$, we obtain
\begin{align*}
(-\Delta)^s u(x_*)
&\le C_{n,s} I_r(x_*)u(x_*)
 +C_{n,s}J_r(x_*)\bigl(u(x_*)-m\bigr),
\end{align*}
where
\[
I_r(x)=\int_{\mathbb R^n\setminus B_r(x_0)}
\frac{\mathrm{d}y}{|x-y|^{n+2s}},
\qquad
J_r(x)=\int_{B_r(x_0)\setminus B_{r/2}(x_0)}
\frac{\mathrm{d}y}{|x-y|^{n+2s}}.
\]
For $x\in B_{r/2}(x_0)$,
\[
I_r(x)\le C_1 r^{-2s},\qquad J_r(x)\ge C_2 r^{-2s},
\]
with $C_1,C_2>0$ depending only on $n,s$.  Evaluating
$\mathscr{L}_su+\Lambda u\ge0$ at $x_*$ therefore gives
\[
\frac{u(x_*)}{m}
\ge
\frac{C_{n,s}C_2}
{C_{n,s}(C_1+C_2)+\Lambda r^{2s}}
=\alpha(n,s,\Lambda r^{2s})>0.
\]
This proves the estimate for smooth supersolutions.

For the distributional case, Lemma~\ref{mollifi} gives
\[
(\mathscr{L}_s+\Lambda)u^\delta\ge0
\quad\text{in }B_{r-\delta}(x_0).
\]
Moreover
\[
u^\delta\ge m
\quad\text{in }
B_{r-\delta}(x_0)\setminus B_{r/2+\delta}(x_0).
\]
Repeating the preceding minimum argument with the radii
$r/2+\delta$ and $r-\delta$ gives, for $0<\delta<r/8$,
\[
u^\delta\ge \alpha_0m
\quad\text{in }B_{r/2+\delta}(x_0),
\]
where
\[
\alpha_0=\alpha_0(n,s,\Lambda r^{2s})>0
\]
is independent of $\delta$.  Letting $\delta\downarrow0$ and using
$u^\delta\to u$ in $L^1_{\rm loc}$ proves the assertion a.e., compare
\cite[Theorem~1]{LWX2018}.
\end{proof}

\begin{proof}[Proof of Theorem~\ref{thm:mixed-antisymmetric}]
Again put $\Lambda=\|c^+\|_{L^\infty(B_r(x_0))}$.  Since $u\ge0$ in
$B_r(x_0)\subset H$,
\[
\mathscr{L}_su+\Lambda u=(\mathscr{L}_su+cu)+(\Lambda-c)u\ge0
\quad\text{in }\mathcal D'(B_r(x_0)\setminus\{x_0\}).
\]
The corresponding ball-version of Theorem~\ref{thm-bocher}, applied with coefficient $\Lambda$, gives
\[
\mathscr{L}_su+\Lambda u\ge0\quad\text{in }\mathcal D'(B_r(x_0)).
\]
Assume first that $u$ is smooth.  For $x,y\in H$ set
\[
K_H(x,y)=\frac1{|x-y|^{n+2s}}-
\frac1{|x-\widetilde y|^{n+2s}}>0.
\]
Using antisymmetry and pairing $H$ with its reflection,
\begin{equation}\label{eq:antisymmetric-fractional-representation}
(-\Delta)^s u(x)
=C_{n,s}\,\mathrm{PV}\int_H (u(x)-u(y))K_H(x,y)\,\mathrm{d}y
+2C_{n,s}u(x)\int_H\frac{\mathrm{d}y}{|x-\widetilde y|^{n+2s}}.
\end{equation}
Let $x_*$ be a minimum of $u$ on $\overline{B_{r/2}(x_0)}$.  If
$u(x_*)<m$, then $x_*$ is interior and $-\Delta u(x_*)\le0$.  Splitting
the first integral in \eqref{eq:antisymmetric-fractional-representation}
into $B_{r/2}(x_0)$, the annulus $B_{r}(x_0)\backslash B_{r/2}(x_0)$, and $H\setminus B_r(x_0)$ gives
\[
(-\Delta)^s u(x_*)
\le C_{n,s}\widehat I_r(x_*)u(x_*)
+C_{n,s}\widehat J_r(x_*)\bigl(u(x_*)-m\bigr),
\]
where
\begin{align*}
\widehat I_r(x)
&=\int_{H\setminus B_r(x_0)}K_H(x,y)\,\mathrm{d}y
 +2\int_H\frac{\mathrm{d}y}{|x-\widetilde y|^{n+2s}},\\
\widehat J_r(x)
&=\int_{B_r(x_0)\setminus B_{r/2}(x_0)}K_H(x,y)\,\mathrm{d}y.
\end{align*}
Since $B_r(x_0)\subset H$, every $x\in B_{r/2}(x_0)$ satisfies
$x_1\ge r/2$.  Hence
$
\widehat I_r(x)\le C_3r^{-2s}.
$
Moreover $\widehat J_r(x)\ge C_4r^{-2s}$.  Indeed, set
\[
E=B_{r/8}(x_0+3re_1/4)
   \subset B_r(x_0)\setminus B_{r/2}(x_0).
\]
For $x\in B_{r/2}(x_0)$ and $y\in E$,
\[
|x-y|\le\frac{11}{8}r,
\qquad
|x-\widetilde y|\ge x_1+|y_1|\ge\frac{17}{8}r,
\]
because $B_r(x_0)\subset H$ implies $(x_0)_1\ge r$.  Consequently
\[
K_H(x,y)\ge
\left[\left(\frac8{11}\right)^{n+2s}
      -\left(\frac8{17}\right)^{n+2s}\right]r^{-n-2s},
\]
and integration over $E$ gives the claimed lower bound.  Evaluating
$\mathscr{L}_su+\Lambda u\ge0$ at $x_*$ gives
\[
\frac{u(x_*)}{m}
\ge
\frac{C_{n,s}C_4}
{C_{n,s}(C_3+C_4)+\Lambda r^{2s}}
=\alpha(n,s,\Lambda r^{2s})>0.
\]

For the distributional case choose $j$ radial and radially
nonincreasing.  Lemma~\ref{mollifi} gives
\[
(\mathscr{L}_s+\Lambda)u^\delta\ge0
\quad\text{in }B_{r-\delta}(x_0),
\]
and antisymmetry gives, for $x\in H$,
\[
u^\delta(x)
=\int_Hu(y)\bigl(j_\delta(x-y)-j_\delta(x-\widetilde y)\bigr)\,\mathrm{d}y\ge0,
\]
since $|x-y|\le|x-\widetilde y|$.  Also
\[
u^\delta\ge m
\quad\text{in }
B_{r-\delta}(x_0)\setminus B_{r/2+\delta}(x_0).
\]
The preceding estimate, with these perturbed radii, is uniform for
$0<\delta<r/8$.  Letting $\delta\downarrow0$ proves the assertion a.e.,
compare \cite[Theorem~2]{LWX2018}.
\end{proof}

\appendix
\section{Auxiliary properties of the mixed operator}

\begin{prop}[Maximum of subsolutions]\label{construct-super}
Let $\Omega\subset\mathbb R^n$ be a domain,
$u,v\in\mathcal L^{2s}(\mathbb R^n)$,
$f,g\in L^1_{\rm loc}(\Omega)$, and $c\in L^\infty(\Omega)$. If
\[
\mathscr{L}_su+cu\le f,
\qquad
\mathscr{L}_sv+cv\le g
\quad\text{in }\mathcal D'(\Omega),
\]
then, for $w=\max\{u,v\}$,
\begin{equation}\label{eq:max-subsolutions}
\mathscr{L}_sw+cw\le
f\chi_{\{u>v\}}+g\chi_{\{u<v\}}
+\max\{f,g\}\chi_{\{u=v\}}
\quad\text{in }\mathcal D'(\Omega).
\end{equation}
\end{prop}

\begin{proof}
Since $|w|\le|u|+|v|$, we have $w\in\mathcal L^{2s}(\mathbb R^n)$. Put $F=f-cu$ and $G=g-cv$, mollify the two inequalities, and let $M_\varepsilon(a,b)=b+\theta_\varepsilon(a-b)$, where $\theta_\varepsilon$ is a smooth convex approximation of $t^+$ and $p_\varepsilon=\theta_\varepsilon'(u^\delta-v^\delta)\in[0,1]$. The local chain rule and the supporting-plane inequality for the fractional term give
\[
\mathscr L_s M_\varepsilon(u^\delta,v^\delta)
\le p_\varepsilon\mathscr L_su^\delta+(1-p_\varepsilon)\mathscr L_sv^\delta
\le p_\varepsilon F^\delta+(1-p_\varepsilon)G^\delta.
\]
Letting $\varepsilon\downarrow0$ and then $\delta\downarrow0$ gives the standard Kato limit (compare \cite[Lemma~2]{LWX2018}): on $\{u>v\}$ and $\{u<v\}$ the weights converge to $1$ and $0$, respectively, while on $\{u=v\}$ every limiting weight lies in $[0,1]$. Hence
\[
\mathscr L_sw\le F\chi_{\{u>v\}}+G\chi_{\{u<v\}}+\max\{F,G\}\chi_{\{u=v\}}.
\]
Adding $cw$ gives \eqref{eq:max-subsolutions}, since $w=u$ on $\{u>v\}$, $w=v$ on $\{u<v\}$, and $u=v=w$ on $\{u=v\}$.
\end{proof}

Let $j\in C_c^\infty(B_1)$ be nonnegative and radial, with
$\int j=1$, and set
\[
j_\delta(x)=\delta^{-n}j(x/\delta),
\qquad
w^\delta=w*j_\delta.
\]

\begin{lemma}[Mollification]\label{mollifi}
Let $w\in\mathcal L^{2s}(\mathbb R^n)$ and
$f\in L^1_{\rm loc}(\Omega)$.  If
\[
\mathscr{L}_sw\le f
\quad\text{in }\mathcal D'(\Omega),
\]
then
\[
\mathscr{L}_sw^\delta\le f^\delta
\quad\text{in }\mathcal D'(\Omega^\delta),
\qquad
\Omega^\delta=\{x\in\mathbb R^n\mid B_\delta(x)\subset\Omega\}.
\]
The same statement holds with $\mathscr{L}_s$ replaced by
$\mathscr{L}_s+c_0$ for any constant $c_0\in\mathbb R$, and the analogous assertions hold with all inequalities reversed.
\end{lemma}

\begin{proof}
For $0\le\varphi\in C_c^\infty(\Omega^\delta)$,
\begin{align*}
\langle\mathscr{L}_sw^\delta-f^\delta,\varphi\rangle
&=\int_{B_\delta}j_\delta(z)
\langle\mathscr{L}_sw-f,\varphi(\cdot+z)\rangle\,\mathrm{d}z\le0.
\end{align*}
The constant zero-order term commutes with convolution. 
For the purely fractional version, compare \cite[Lemma~1]{LWX2018}.
\end{proof}

\begin{theorem}[Interior smoothness]\label{smooth-harmonic}
Let $\Omega\subset\mathbb R^n$ be a domain and let
$w\in\mathcal L^{2s}(\mathbb R^n)$. If
$\mathscr{L}_sw=0$ in $\mathcal D'(\Omega)$, then $w\in C^\infty(\Omega)$.
\end{theorem}

\begin{proof}
Fix smooth subdomains $U\Subset V\Subset\Omega$ and choose $\chi\in C_c^\infty(\Omega)$
with $\chi=1$ in a neighborhood of $\overline V$. Since
$\chi w\in L_c^1(\mathbb R^n)$, one has
$\chi w\in H^{-N}(\mathbb R^n)$ for every $N>n/2$. For $x\in V$, in the
distributional sense,
\begin{equation}\label{eq:frac-localization}
(-\Delta)^sw
=(-\Delta)^s(\chi w)-g_\chi,
\qquad
g_\chi(x)=C_{n,s}\int_{\mathbb R^n}
\frac{(1-\chi(y))w(y)}{|x-y|^{n+2s}}\,\mathrm{d}y.
\end{equation}
For every multi-index $\gamma$ and every $x\in V$,
\[
\partial_x^\gamma g_\chi(x)
=C_{n,s}\int_{\mathbb R^n}(1-\chi(y))w(y)
\partial_x^\gamma|x-y|^{-n-2s}\,\mathrm{d}y.
\]
The distance between $V$ and $\operatorname{supp}(1-\chi)$ is positive,
and the tail assumption on $w$ makes the last integral absolutely
convergent after any number of $x$-derivatives. Hence $g_\chi\in C^\infty(V)$.

Assume now that \(w\in H^t_{\rm loc}(\Omega)\) for some
\(t\in\mathbb R\). Since \(\chi\in C_c^\infty(\Omega)\), we have $\chi w\in H^t(\mathbb R^n)$. The Fourier multiplier representation of \((-\Delta)^s\) gives
\begin{equation}\label{eq:Hs-frac-map}
\|(-\Delta)^s(\chi w)\|_{H^{t-2s}(\mathbb R^n)}
\leq C\|\chi w\|_{H^t(\mathbb R^n)}.
\end{equation}
Since \(g_\chi\in C^\infty(V)\), 
\eqref{eq:Hs-frac-map} and \eqref{eq:frac-localization} imply
$
(-\Delta)^s w\in H^{t-2s}_{\rm loc}(V).
$
Using \(\mathscr L_s w=0\) in \(\mathcal D'(\Omega)\), we obtain
$
-\Delta w=-(-\Delta)^s w
\in H^{t-2s}_{\rm loc}(V).
$
The local elliptic regularity theorem for the Laplacian yields
$
w\in H^{t+2-2s}_{\rm loc}(V).
$
Since \(U\Subset V\Subset\Omega\) are arbitrary, we have proved the
bootstrap implication
\begin{equation}\label{eq:sobolev-bootstrap-mixed}
w\in H^t_{\rm loc}(\Omega)
\quad\Longrightarrow\quad
w\in H^{t+2-2s}_{\rm loc}(\Omega).
\end{equation}

The argument at the beginning of the proof, applied to arbitrary
compactly supported cutoffs, shows that
\[
w\in H^{-N}_{\rm loc}(\Omega)
\qquad\text{for every }N>\frac n2.
\]
Iterating \eqref{eq:sobolev-bootstrap-mixed}, we obtain
\[
w\in H^{-N+k(2-2s)}_{\rm loc}(\Omega)
\qquad\text{for every }k\in\mathbb N.
\]
The Sobolev embedding theorem then gives \(w\in C^m(U)\) for every $m\in\mathbb N$.
Since \(U\Subset\Omega\) are arbitrary, it follows that
\(w\in C^\infty(\Omega)\).
\end{proof}

\section{Homogeneous Distributions}
\label{app:homogeneous-distributions}

\begin{prop}[Fourier transforms of radial homogeneous distributions]
\label{prop:homogeneous-fourier-transform}
Under the normalization \eqref{eq:fourier-normalization}, the following
statements hold.

\begin{enumerate}

\item For \(0<\alpha<n\),
\begin{equation}
    \mathcal F^{-1}(|\xi|^{-\alpha})(x)
    =
    \kappa_{n,\alpha}|x|^{\alpha-n}.
    \label{eq:riesz-kernel-formula}
\end{equation}
The family \(|\xi|^{-\alpha}\), initially defined as a tempered
distribution for \(\Re\alpha<n\), admits a meromorphic continuation
to \(\alpha\in\mathbb C\), with possible poles at
$\alpha=n,n+2,n+4,\cdots$.
The identity \eqref{eq:riesz-kernel-formula} extends accordingly by
meromorphic continuation. Away from the origin, this continuation is
represented by
$\kappa_{n,\alpha}|x|^{\alpha-n}$.
In particular, if \(\alpha=-2k\), \(k\in\mathbb N_0\), then
\[
    \mathcal F^{-1}(|\xi|^{-\alpha})
    =
    (-\Delta)^k\delta_0.
\]

\item Define
\begin{equation*}
\begin{aligned}
    \left\langle\operatorname{Fp}|\xi|^{-n},\varphi\right\rangle
    &=
    \int_{|\xi|<1}
    \frac{\varphi(\xi)-\varphi(0)}{|\xi|^n}\,\mathrm{d}\xi\\
    &\quad+
    \int_{|\xi|>1}
    \frac{\varphi(\xi)}{|\xi|^n}\,\mathrm{d}\xi,
    \qquad
    \varphi\in\mathcal S(\mathbb R^n).
\end{aligned}
\end{equation*}
Then
\begin{equation}
    \mathcal F^{-1}
    \bigl(\operatorname{Fp}|\xi|^{-n}\bigr)
    =
    c_n\log\frac1{|x|}
    +
    c_n\left(
        \log2+\frac12\psi\!\left(\frac n2\right)
        -\frac{\gamma}{2}
    \right)
    \quad\text{in }\mathcal S'(\mathbb R^n),
    \label{eq:finite-part-fourier-transform}
\end{equation}
where
\[
    c_n
    =
    \frac{2^{1-n}}{\pi^{n/2}\Gamma(n/2)},
\]
\(\gamma\) is Euler's constant, and
\(\psi=\Gamma'/\Gamma\).

\end{enumerate}
\end{prop}

\begin{proof}
For \(0<\alpha<n\),
\begin{equation*}
    |\xi|^{-\alpha}
    =
    \frac{1}{\Gamma(\alpha/2)}
    \int_0^\infty
    t^{\alpha/2-1}e^{-t|\xi|^2}\,\mathrm{d}t,
\end{equation*}
and
\begin{equation*}
    \mathcal F^{-1}(e^{-t|\xi|^2})(x)
    =
    (4\pi t)^{-n/2}e^{-|x|^2/(4t)}.
\end{equation*}
Hence, for \(x\neq0\),
\begin{align*}
    \mathcal F^{-1}(|\xi|^{-\alpha})(x)
    &=
    \frac{(4\pi)^{-n/2}}{\Gamma(\alpha/2)}
    \int_0^\infty
    t^{(\alpha-n)/2-1}
    e^{-|x|^2/(4t)}\,\mathrm{d}t\\
    &=
    \frac{
    \Gamma\!\left(\frac{n-\alpha}{2}\right)}
    {2^\alpha\pi^{n/2}\Gamma(\alpha/2)}
    |x|^{\alpha-n}.
\end{align*}
This proves \eqref{eq:riesz-kernel-formula} for \(0<\alpha<n\).

We next construct the meromorphic continuation. Fix \(N\ge0\), and
let \(P_N\varphi\) denote the Taylor polynomial of \(\varphi\) at the
origin of degree \(N\). Initially, for \(\Re\alpha<n\),
\begin{align*}
    \int_{\mathbb R^n}
    |\xi|^{-\alpha}\varphi(\xi)\,\mathrm{d}\xi
    &=
    \int_{|\xi|<1}
    |\xi|^{-\alpha}
    \bigl(\varphi(\xi)-P_N\varphi(\xi)\bigr)\,\mathrm{d}\xi+
    \int_{|\xi|>1}
    |\xi|^{-\alpha}\varphi(\xi)\,\mathrm{d}\xi\\
    &\qquad+
    \sum_{|\rho|\le N}
    \frac{\partial^\rho\varphi(0)}{\rho!}
    \int_{|\xi|<1}
    |\xi|^{-\alpha}\xi^\rho\,\mathrm{d}\xi.
\end{align*}
For
$\Re\alpha<n+|\rho|$,
polar coordinates give
\begin{equation*}
    \int_{|\xi|<1}
    |\xi|^{-\alpha}\xi^\rho\,\mathrm{d}\xi
    =
    \frac{1}{n+|\rho|-\alpha}
    \int_{\mathbb S^{n-1}}
    \theta^\rho\,\mathrm{d}\theta.
\end{equation*}
The first integral above, containing
\(\varphi-P_N\varphi\), is holomorphic for
$\Re\alpha<n+N+1$,
since
\[
    \varphi(\xi)-P_N\varphi(\xi)
    =
    O(|\xi|^{N+1})
    \qquad\text{as }\xi\to0.
\]
The integral over \(|\xi|>1\) is entire in \(\alpha\), because
\(\varphi\) is a Schwartz function. It follows that, for
\(\Re\alpha<n+N+1\), the continuation is given by
\begin{align*}
    \langle T_\alpha,\varphi\rangle
    &=
    \int_{|\xi|<1}
    |\xi|^{-\alpha}
    \bigl(\varphi(\xi)-P_N\varphi(\xi)\bigr)\,\mathrm{d}\xi+
    \int_{|\xi|>1}
    |\xi|^{-\alpha}\varphi(\xi)\,\mathrm{d}\xi\\
    &\qquad+
    \sum_{|\rho|\le N}
    \frac{\partial^\rho\varphi(0)}{\rho!}
    \frac{1}{n+|\rho|-\alpha}
    \int_{\mathbb S^{n-1}}
    \theta^\rho\,\mathrm{d}\theta.
\end{align*}
Thus the only possible poles in this half-plane are
$n,n+1,\ldots,n+N$.
If \(|\rho|\) is odd, then
$\int_{\mathbb S^{n-1}}\theta^\rho\,\mathrm{d}\theta=0$,
so the possible poles reduce to $n,n+2,n+4,\cdots$.
The continuations corresponding to different values of \(N\) agree
on their common domains and hence define an
\(\mathcal S'\)-valued meromorphic family \(T_\alpha\) on
\(\mathbb C\).

For \(0<\alpha<n\), the computation above gives
$\mathcal F^{-1}T_\alpha
    =
    \kappa_{n,\alpha}|x|^{\alpha-n}$.
Since the inverse Fourier transform is continuous on
\(\mathcal S'\), uniqueness of meromorphic continuation extends this
identity to all noncritical values of \(\alpha\). If
\(\alpha=-2k\), then
$T_{-2k}=|\xi|^{2k}$,
and therefore
$\mathcal F^{-1}T_{-2k}
    =
    (-\Delta)^k\delta_0$.

We now consider the first critical value \(\alpha=n\).
For \(\varepsilon>0\) and
\(\varphi\in\mathcal S(\mathbb R^n)\), splitting at \(|\xi|=1\) and
subtracting \(\varphi(0)\) on the unit ball gives
\begin{equation}
    |\xi|^{-n+\varepsilon}
    =
    \frac{\omega_{n-1}}{\varepsilon}\delta_0
    +
    \operatorname{Fp}|\xi|^{-n}
    +
    \varepsilon R_\varepsilon
    \qquad\text{in }\mathcal S'(\mathbb R^n),
    \label{eq:finite-part-laurent-expansion}
\end{equation}
where
\begin{equation}
\begin{aligned}
    \langle R_\varepsilon,\varphi\rangle
    &=
    \int_{|\xi|<1}
    \frac{\varphi(\xi)-\varphi(0)}{|\xi|^n}
    \frac{|\xi|^\varepsilon-1}{\varepsilon}\,\mathrm{d}\xi+
    \int_{|\xi|>1}
    \frac{\varphi(\xi)}{|\xi|^n}
    \frac{|\xi|^\varepsilon-1}{\varepsilon}\,\mathrm{d}\xi.
\end{aligned}
    \label{eq:finite-part-remainder}
\end{equation}
Indeed,
\begin{equation*}
    \int_{|\xi|<1}
    |\xi|^{-n+\varepsilon}\,\mathrm{d}\xi
    =
    \omega_{n-1}
    \int_0^1r^{-1+\varepsilon}\,\mathrm{d}r
    =
    \frac{\omega_{n-1}}{\varepsilon}.
\end{equation*}

Fix \(0<\varepsilon\le\varepsilon_0\). The mean-value theorem gives
\begin{equation*}
    \left|
    \frac{r^\varepsilon-1}{\varepsilon}
    \right|
    \le|\log r|,
    \quad 0<r<1,
\quad
\text{and}\quad
    \left|
    \frac{r^\varepsilon-1}{\varepsilon}
    \right|
    \le r^{\varepsilon_0}\log r,
    \quad r>1.
\end{equation*}
Since
\[
    |\varphi(\xi)-\varphi(0)|
    \le C|\xi|
\]
near the origin, the first integral in
\eqref{eq:finite-part-remainder} is bounded by a constant multiple of
\begin{equation*}
    \int_{|\xi|<1}
    |\xi|^{1-n}|\log|\xi||\,\mathrm{d}\xi<\infty.
\end{equation*}
The second integral is bounded by a fixed Schwartz seminorm of
\(\varphi\). Hence \(R_\varepsilon\) remains bounded in
\(\mathcal S'\).

Taking inverse Fourier transforms in
\eqref{eq:finite-part-laurent-expansion} and using
$\mathcal F^{-1}\delta_0=(2\pi)^{-n}$,
we obtain
\begin{equation}
    \mathcal F^{-1}(|\xi|^{-n+\varepsilon})
    =
    \frac{c_n}{\varepsilon}
    +
    \mathcal F^{-1}
    \bigl(\operatorname{Fp}|\xi|^{-n}\bigr)
    +
    O_{\mathcal S'}(\varepsilon),
    \label{eq:finite-part-inverse-laurent-expansion}
\end{equation}
where
\[
    c_n
    =
    \omega_{n-1}(2\pi)^{-n}
    =
    \frac{2^{1-n}}{\pi^{n/2}\Gamma(n/2)}.
\]

On the other hand, \eqref{eq:riesz-kernel-formula} with
\(\alpha=n-\varepsilon\) gives
\begin{equation*}
    \mathcal F^{-1}(|\xi|^{-n+\varepsilon})(x)
    =
    \frac{\Gamma(\varepsilon/2)}
    {2^{n-\varepsilon}\pi^{n/2}
     \Gamma\!\left(\frac{n-\varepsilon}{2}\right)}
    |x|^{-\varepsilon}.
\end{equation*}
We use
\begin{align*}
    &\Gamma(\varepsilon/2)
    =
    \frac{2}{\varepsilon}-\gamma+O(\varepsilon),
    \qquad\quad
    2^\varepsilon
    =
    1+\varepsilon\log2+O(\varepsilon^2),\\&
    \frac{1}
    {\Gamma\!\left(\frac{n-\varepsilon}{2}\right)}
    =
    \frac{1}{\Gamma(n/2)}
    \left(
        1+\frac{\varepsilon}{2}
        \psi\!\left(\frac n2\right)
        +O(\varepsilon^2)
    \right).
\end{align*}

It remains to justify the expansion of
\(|x|^{-\varepsilon}\) in \(\mathcal S'\).
Choose \(0<\varepsilon_0<n\). Taylor's formula gives, uniformly for
\(0<\varepsilon\le\varepsilon_0\),
\begin{equation*}
\left|
\frac{|x|^{-\varepsilon}
      -1+\varepsilon\log|x|}
     {\varepsilon^2}
\right|
\le
C
\begin{cases}
|x|^{-\varepsilon_0}
\bigl(1+|\log|x||^2\bigr),
&|x|<1,\\[1ex]
1+|\log|x||^2,
&|x|\ge1.
\end{cases}
\end{equation*}
The right-hand side is locally integrable and has at most logarithmic
growth at infinity. Hence
\[
    |x|^{-\varepsilon}
    =
    1-\varepsilon\log|x|
    +O_{\mathcal S'}(\varepsilon^2).
\]
Combining the preceding expansions yields
\begin{align*}
    \mathcal F^{-1}(|\xi|^{-n+\varepsilon})
    &=
    \frac{c_n}{\varepsilon}
    +
    c_n\log\frac1{|x|}+
    c_n\left(
        \log2+\frac12\psi\!\left(\frac n2\right)
        -\frac{\gamma}{2}
    \right)
    +
    O_{\mathcal S'}(\varepsilon).
\end{align*}
Comparison with
\eqref{eq:finite-part-inverse-laurent-expansion}
gives \eqref{eq:finite-part-fourier-transform}.
\end{proof}

\section{Cutoff Multipliers and Dyadic Kernel Estimates}
\label{app:dyadic-estimates}

\begin{lemma}[Cutoff and symbol estimates]
\label{lem:cutoff-symbol-estimates}
\quad
\begin{enumerate}
\item Suppose that \(a\in C^\infty(\mathbb R^n)\) vanishes near the
origin and that, for some \(\mu\in\mathbb R\) and every multi-index \(\rho\),
\begin{equation}
    |\partial_\xi^\rho a(\xi)|
    \leq C_\rho|\xi|^{-\mu-|\rho|}
    \qquad\text{for any }|\xi|\geq1.
    \label{eq:high-frequency-symbol-assumption}
\end{equation}
Then \(K=\mathcal F^{-1}a\) is smooth on
\(\mathbb R^n\setminus\{0\}\), is rapidly decreasing at infinity, and
for every multi-index \(\gamma\) and \(0<|x|\leq1\),
\begin{equation}
    |\partial_x^\gamma K(x)|
    \leq
    \begin{cases}
        C_\gamma|x|^{\mu-n-|\gamma|},&\mu<n+|\gamma|,\\
        C_\gamma(1+|\log|x||),&\mu=n+|\gamma|,\\
        C_\gamma,&\mu>n+|\gamma|.
    \end{cases}
    \label{eq:high-frequency-kernel-bound}
\end{equation}
If, in addition, \(a\) is radial, \(\mu>n\), and
\begin{equation}
    0<\beta<\min\{\mu-n,2\},
    \label{eq:radial-regularity-range}
\end{equation}
then
\begin{equation}
    K(x)=K(0)+o(|x|^\beta),
    \qquad
    \nabla K(x)=o(|x|^{\beta-1}),
    \qquad\text{as }|x|\to0.
    \label{eq:radial-high-frequency-regularity}
\end{equation}

\item Suppose that \(b\) is compactly supported, smooth on
\(\mathbb R^n\setminus\{0\}\), and that, for some \(\nu>-n\) and every multi-index \(\rho\),
\begin{equation*}
    |\partial_\xi^\rho b(\xi)|
    \leq C_\rho|\xi|^{\nu-|\rho|}
    \qquad\text{for any }0<|\xi|\leq2.
\end{equation*}
Then, for every multi-index \(\gamma\) and \(|x|\geq1\),
\begin{equation}
    |\partial_x^\gamma\mathcal F^{-1}b(x)|
    \leq C_\gamma|x|^{-n-\nu-|\gamma|}.
    \label{eq:low-frequency-kernel-bound}
\end{equation}
\end{enumerate}
\end{lemma}

\begin{proof}
Choose \(\psi\in C_c^\infty(\mathbb R^n\setminus\{0\})\), supported in
\(\{\frac12\leq|\xi|\leq2\}\), such that
\begin{equation*}
    \sum_{j\in\mathbb Z}\psi(2^{-j}\xi)=1
    \qquad\text{for any }\xi\neq0.
\end{equation*}

\smallskip
\emph{High frequencies.}
Because \(a\) vanishes near the origin, after a harmless shift of the
index,
\begin{equation*}
    a=\sum_{j=0}^\infty a_j
    \quad\text{in }\mathcal S',
    \qquad
    a_j(\xi)=\psi(2^{-j}\xi)a(\xi).
\end{equation*}
Each \(a_j\) is supported where \(|\xi|\asymp2^j\). For a multi-index
\(\gamma\),
\begin{equation*}
    K_{j,\gamma}(x)
    =\partial_x^\gamma\mathcal F^{-1}a_j(x)
    =\frac{1}{(2\pi)^n}
      \int e^{ix\cdot\xi}(i\xi)^\gamma a_j(\xi)\,\mathrm{d}\xi.
\end{equation*}
After \(\xi=2^j\zeta\),
\begin{equation*}
    K_{j,\gamma}(x)
    =2^{j(n+|\gamma|-\mu)}
      \frac{1}{(2\pi)^n}
      \int e^{i2^jx\cdot\zeta}A_{j,\gamma}(\zeta)\,\mathrm{d}\zeta,
\end{equation*}
where
\begin{equation*}
    A_{j,\gamma}(\zeta)
    =2^{j\mu}(i\zeta)^\gamma\psi(\zeta)a(2^j\zeta).
\end{equation*}
The amplitudes \(A_{j,\gamma}\) are supported in one fixed annulus and
are uniformly bounded in every \(C^M\)-norm. Indeed, every derivative
of \(a(2^j\zeta)\) contributes a factor \(2^j\), which is canceled by
the decay in \eqref{eq:high-frequency-symbol-assumption}.

For an integer \(L\geq0\),
\begin{equation*}
    (1-\Delta_\zeta)^L e^{i2^jx\cdot\zeta}
    =(1+2^{2j}|x|^2)^L e^{i2^jx\cdot\zeta}.
\end{equation*}
Since the amplitude is compactly supported, integration by parts gives
\begin{equation}
    |K_{j,\gamma}(x)|
    \leq C_{N,\gamma}
    2^{j(n+|\gamma|-\mu)}(1+2^j|x|)^{-N}
    \label{eq:high-frequency-dyadic-estimate}
\end{equation}
for every \(N\geq0\). This is first obtained for even integers \(N=2L\),
and the remaining values follow by weakening the estimate.

For every compact set \(E\Subset\mathbb R^n\setminus\{0\}\), choosing
\(N>n+|\gamma|-\mu\) makes
\(\sum_jK_{j,\gamma}\) uniformly convergent on \(E\). Hence the dyadic
sum is smooth away from the origin. Since the partial sums converge to
\(a\) in \(\mathcal S'\), this smooth function represents
\(\mathcal F^{-1}a\) there.

Let \(0<|x|\leq1\), choose \(J\geq0\) such that
\(2^J|x|\leq1<2^{J+1}|x|\), and put
\(A=n+|\gamma|-\mu\). For \(j\leq J\),
\begin{equation*}
    \sum_{j=0}^J|K_{j,\gamma}(x)|
    \leq C\sum_{j=0}^J2^{jA}
    \leq
    \begin{cases}
        C|x|^{-A},&A>0,\\
        C(1+|\log|x||),&A=0,\\
        C,&A<0.
    \end{cases}
\end{equation*}
For \(j>J\), choose \(N>\max\{A,0\}\). Then
\begin{align*}
    \sum_{j>J}|K_{j,\gamma}(x)|
    &\leq C|x|^{-N}\sum_{j>J}2^{j(A-N)}\\
    &\leq C|x|^{-N}2^{J(A-N)}
     \leq C|x|^{-A}.
\end{align*}
If \(A\leq0\), the last expression is bounded. This proves
\eqref{eq:high-frequency-kernel-bound}.

For \(|x|\geq1\), choose
\(N>M+\max\{A,0\}\) in
\eqref{eq:high-frequency-dyadic-estimate}. Then
\begin{equation*}
    \sum_{j=0}^\infty|K_{j,\gamma}(x)|
    \leq C|x|^{-N}\sum_{j=0}^\infty2^{j(A-N)}
    \leq C_{M,\gamma}|x|^{-M}.
\end{equation*}
Thus \(K\) and all its derivatives are rapidly decreasing at infinity.

Assume now that \(a\) is radial, \(\mu>n\), and
\eqref{eq:radial-regularity-range} holds. Then \(a\in L^1\), and
\begin{equation*}
    K(x)=\frac{1}{(2\pi)^n}\int e^{ix\cdot\xi}a(\xi)\,\mathrm{d}\xi.
\end{equation*}
Choose \(\beta<\beta'<\min\{\mu-n,2\}\). Then \(\int|\xi|^{\beta'}|a(\xi)|\,\mathrm{d}\xi<\infty\). If
\(0<\beta'\leq1\), use \(|e^{it}-1|\leq C|t|^{\beta'}\). If
\(1<\beta'<2\), radiality implies that \(a\) is even and
\(\int\xi a(\xi)\,\mathrm{d}\xi=0\), use
\(|e^{it}-1-it|\leq C|t|^{\beta'}\). In both cases,
\begin{equation*}
    |K(x)-K(0)|
    \leq C|x|^{\beta'}\int|\xi|^{\beta'}|a(\xi)|\,\mathrm{d}\xi
    =o(|x|^\beta).
\end{equation*}

For the gradient, suppose first that \(0<\beta<1\). If
\(\mu<n+1\), \eqref{eq:high-frequency-kernel-bound} gives
\begin{equation*}
    |\nabla K(x)|\leq C|x|^{\mu-n-1}=o(|x|^{\beta-1}),
\end{equation*}
because \(\beta<\mu-n\). If \(\mu=n+1\), the logarithmic bound has
the same conclusion. If \(\mu>n+1\), boundedness suffices. If
\(\beta=1\), then \(\mu>n+1\) and \(\xi a\in L^1\). Hence
\(\nabla K\) is continuous and \(\nabla K(0)=0\) by radiality. If
\(1<\beta<2\), then
\begin{align*}
    |\nabla K(x)-\nabla K(0)|
    &\leq C\int|\xi|\,|e^{ix\cdot\xi}-1|\,|a(\xi)|\,\mathrm{d}\xi\\
    &\leq C|x|^{\beta'-1}\int|\xi|^{\beta'}|a(\xi)|\,\mathrm{d}\xi
     =o(|x|^{\beta-1}).
\end{align*}
This proves \eqref{eq:radial-high-frequency-regularity}.

\smallskip
\emph{Low frequencies.}
Since \(\nu>-n\), the function \(b\) is locally integrable near the
origin. Decompose
\begin{equation*}
    b=b_*+\sum_{j=0}^\infty b_j,
    \qquad
    b_j(\xi)=\psi(2^j\xi)b(\xi),
\end{equation*}
where \(b_*\in C_c^\infty(\mathbb R^n\setminus\{0\})\) collects the
finitely many annuli away from the origin. The inverse Fourier transform
of \(b_*\) is Schwartz. The support of \(b_j\) lies where
\(|\xi|\asymp2^{-j}\). Rescaling \(\xi=2^{-j}\zeta\) and repeating
the previous integration-by-parts argument gives
\begin{equation*}
    |\partial_x^\gamma\mathcal F^{-1}b_j(x)|
    \leq C_{N,\gamma}
    2^{-j(n+\nu+|\gamma|)}(1+2^{-j}|x|)^{-N}.
\end{equation*}
Let \(|x|\geq1\), choose \(J\geq0\) such that
\(2^J\leq|x|<2^{J+1}\), and set \(B=n+\nu+|\gamma|>0\). For
\(j\geq J\),
\begin{equation*}
    \sum_{j\geq J}|\partial_x^\gamma\mathcal F^{-1}b_j(x)|
    \leq C\sum_{j\geq J}2^{-jB}
    \leq C|x|^{-B}.
\end{equation*}
For \(j<J\), choose \(N>B\). Then
\begin{align*}
    \sum_{j<J}|\partial_x^\gamma\mathcal F^{-1}b_j(x)|
    &\leq C|x|^{-N}\sum_{j<J}2^{j(N-B)}\\
    &\leq C|x|^{-N}2^{J(N-B)}
     \leq C|x|^{-B}.
\end{align*}
Adding the Schwartz contribution from \(b_*\) proves
\eqref{eq:low-frequency-kernel-bound}.
\end{proof}

\begin{corollary}[Effect of a frequency cutoff]
\label{cor:cutoff-homogeneous-terms}
Let \(\chi\in C_c^\infty(\mathbb R^n)\) be radial and equal to \(1\)
near the origin, and set \(\eta=1-\chi\).

\begin{enumerate}

\item If \(\alpha>0\) and
\(\alpha\notin\{n,n+2,n+4,\ldots\}\), then
\begin{equation}
    \mathcal F^{-1}(\eta|\xi|^{-\alpha})(x)
    =\kappa_{n,\alpha}|x|^{\alpha-n}+H_\alpha(x),
    \label{eq:high-frequency-cutoff-homogeneous}
\end{equation}
where \(H_\alpha\) is smooth and radial near the origin. At
\(\alpha=n\),
\begin{equation}
    \mathcal F^{-1}
    \bigl(\eta\operatorname{Fp}|\xi|^{-n}\bigr)(x)
    =c_n\log\frac1{|x|}+H_n(x),
    \label{eq:high-frequency-cutoff-critical}
\end{equation}
where \(H_n\) is smooth and radial near the origin.

\item If \(\alpha<n\) and
\(\alpha\notin\{0,-2,-4,\ldots\}\), then, for every \(N>0\),
\begin{equation}
    \mathcal F^{-1}(\chi|\xi|^{-\alpha})(x)
    =\kappa_{n,\alpha}|x|^{\alpha-n}
    +O(|x|^{-N}),
    \qquad |x|\to\infty.
    \label{eq:low-frequency-cutoff-homogeneous}
\end{equation}
If \(\alpha\in\{0,-2,-4,\ldots\}\), then
\(\mathcal F^{-1}(\chi|\xi|^{-\alpha})\) is a Schwartz function.

\end{enumerate}
\end{corollary}

\begin{proof}
Let \(T_\alpha\) denote the homogeneous continuation of
\(|\xi|^{-\alpha}\) given by
Proposition~\ref{prop:homogeneous-fourier-transform}. For the first
assertion, since
\(\eta T_\alpha=T_\alpha-\chi T_\alpha\) and \(\chi T_\alpha\) is a
compactly supported distribution, its inverse Fourier transform is
smooth. Proposition~\ref{prop:homogeneous-fourier-transform} therefore
gives \eqref{eq:high-frequency-cutoff-homogeneous}. The smooth remainder
is radial because both \(T_\alpha\) and \(\chi\) are radial.

At \(\alpha=n\), the same argument applied to
\(\operatorname{Fp}|\xi|^{-n}\) gives
\eqref{eq:high-frequency-cutoff-critical}.

For the second assertion, assume that \(\alpha<n\) and
\(\alpha\notin\{0,-2,-4,\ldots\}\). In this range \(T_\alpha\) agrees
with the locally integrable function \(|\xi|^{-\alpha}\). Since
\(\chi T_\alpha=T_\alpha-\eta T_\alpha\),
Proposition~\ref{prop:homogeneous-fourier-transform} gives
\(\mathcal F^{-1}T_\alpha(x)=\kappa_{n,\alpha}|x|^{\alpha-n}\) for
\(x\ne0\). Moreover, \(\eta T_\alpha=\eta|\xi|^{-\alpha}\) is smooth and satisfies
the high-frequency symbol estimates in
Lemma~\ref{lem:cutoff-symbol-estimates}, with \(\mu=\alpha\). Hence
\(\mathcal F^{-1}(\eta T_\alpha)\) is rapidly decreasing at infinity,
which proves \eqref{eq:low-frequency-cutoff-homogeneous}.

Finally, if \(\alpha=-2k\), \(k\in\mathbb N_0\), then
\(\chi|\xi|^{-\alpha}=\chi|\xi|^{2k}\in C_c^\infty(\mathbb R^n)\),
so its inverse Fourier transform is Schwartz.
\end{proof}

{\bf{Acknowledgments.}} 
C.~Li and J.~Xie are partially supported by the National Natural Science Foundation of China (Grant No. W2531006, 12250710674 and 12031012) and the Institute of Modern Analysis-A Frontier Research Center of Shanghai. L.~Wang was funded by the European Union through the European Research
Council (ERC), under the Starting Grant ``ANGEVA'' (grant agreement
No.~101076411).  Views and opinions expressed are, however, those of the
author only and do not necessarily reflect those of the European Union or the European Research Council.  Neither the European Union nor the granting
authority can be held responsible for them. 

\medskip
{\bf AI disclosure statement.}
During the preparation of this manuscript, the authors used ChatGPT-5.6 Sol Plus as an interactive generative-AI tool to assist with parts of
the asymptotic analysis and the organization of mathematical arguments,
as well as with editing and presentation. All mathematical results and
proofs were reviewed and verified by the authors, who take full
responsibility for the content of the manuscript.

\medskip
{\bf Data availability statement.} No data were used for the research described in the article.

\medskip
{\bf Conflict of interest statement.} On behalf of all authors, the corresponding author states that there is no conflict of interest.


\begin{thebibliography}{99}

\bibitem{ABR1992}
S. Axler, P. Bourdon, and W. Ramey,
\emph{B\^ocher's theorem},
Amer. Math. Monthly \textbf{99} (1992), 51--55.


\bibitem{AJS2018}
N. Abatangelo, S. Jarohs, and A. Salda\~na,
\emph{Green function and Martin kernel for higher-order fractional Laplacians in balls},
Nonlinear Anal. \textbf{175} (2018), 173--190.


\bibitem{Bocher1903}
M. B\^ocher,
\emph{Singular points of functions which satisfy partial differential equations of the elliptic type},
Bull. Amer. Math. Soc. \textbf{9} (1903), 455--465.



\bibitem{BL1981}
H. Brezis and P.-L. Lions,
\emph{A note on isolated singularities for linear elliptic equations},
in \emph{Mathematical Analysis and Applications, Part A}, L. Nachbin, ed.,
Adv. in Math. Suppl. Stud., vol. 7A, Academic Press, New York, 1981,
263--266.


\bibitem{BI2008}
G. Barles and C. Imbert,
\emph{Second-order elliptic integro-differential equations. Viscosity solutions' theory revisited},
Ann. Inst. H. Poincar\'e C Anal. Non Lin\'eaire \textbf{25} (2008), 567--585.

\bibitem{BDQ2025}
B. Barrios, L.~M. Del Pezzo, and A. Quaas,
\emph{Mixed local and nonlocal Laplacian without standard critical exponent for Lane--Emden equation},
arXiv preprint 2507.12258, 2025.

\bibitem{BDVV2022}
S. Biagi, S. Dipierro, E. Valdinoci, and E. Vecchi,
\emph{Mixed local and nonlocal elliptic operators: regularity and maximum principles},
Comm. Partial Differential Equations \textbf{47} (2022), 585--629.

\bibitem{BDVV2023}
S. Biagi, S. Dipierro, E. Valdinoci, and E. Vecchi,
\emph{A Faber--Krahn inequality for mixed local and nonlocal operators},
J. Anal. Math. \textbf{150} (2023), 405--448.

\bibitem{BMV2024}
S. Biagi, D. Mugnai, and E. Vecchi,
\emph{A Brezis--Oswald approach for mixed local and nonlocal operators},
Commun. Contemp. Math. \textbf{26} (2024), Article 2250057.


\bibitem{BVDV2021}
S. Biagi, E. Vecchi, S. Dipierro, and E. Valdinoci,
\emph{Semilinear elliptic equations involving mixed local and nonlocal operators},
Proc. Roy. Soc. Edinburgh Sect. A \textbf{151} (2021), 1611--1641.

\bibitem{BJK2010}
I.~H. Biswas, E.~R. Jakobsen, and K.~H. Karlsen,
\emph{Viscosity solutions for a system of integro-PDEs and connections to optimal switching and control of jump-diffusion processes},
Appl. Math. Optim. \textbf{62} (2010), 47--80.

\bibitem{BK2023}
A. Biswas and S. Khan,
\emph{Existence-uniqueness for nonlinear integro-differential equations with drift in $\mathbb R^d$},
SIAM J. Math. Anal. \textbf{55} (2023), 4378--4409.

\bibitem{BG1960}
R.~M. Blumenthal and R.~K. Getoor,
\emph{Some theorems on stable processes},
Trans. Amer. Math. Soc. \textbf{95} (1960), 263--273.

\bibitem{Bdd2022}
S. Buccheri, J.~V. da Silva, and L.~H. de Miranda,
\emph{A system of local/nonlocal $p$-Laplacians: the eigenvalue problem and its asymptotic limit as $p\to\infty$},
Asymptot. Anal. \textbf{128} (2022), 149--181.

\bibitem{CKS2011}
Z.-Q. Chen, P. Kim, and R. Song,
\emph{Heat kernel estimates for $\Delta+\Delta^{\alpha/2}$ in $C^{1,1}$ open sets},
J. Lond. Math. Soc. (2) \textbf{84} (2011), 58--80.

\bibitem{CKSV2010}
Z.-Q. Chen, P. Kim, R. Song, and Z. Vondra\v cek,
\emph{Sharp Green function estimates for $\Delta+\Delta^{\alpha/2}$ in $C^{1,1}$ open sets and their applications},
Illinois J. Math. \textbf{54} (2010), 981--1024.

\bibitem{CLL2017}
W. Chen, C. Li, and Y. Li,
\emph{A direct method of moving planes for the fractional Laplacian},
Adv. Math. \textbf{308} (2017), 404--437.

\bibitem{CLM2020}
W. Chen, Y. Li, and P. Ma,
\emph{The Fractional Laplacian},
World Scientific, Hackensack, NJ, 2020.

\bibitem{DM2024}
C. De Filippis and G. Mingione,
\emph{Gradient regularity in mixed local and nonlocal problems},
Math. Ann. \textbf{388} (2024), 261--328.

\bibitem{DQ2026}
L.~M. Del Pezzo and A. Quaas,
\emph{The fundamental solution of the fractional $p$-Laplacian},
NoDEA Nonlinear Differential Equations Appl. \textbf{33} (2026), Article 62.


\bibitem{DPV2022}
S. Dipierro, E. Proietti Lippi, and E. Valdinoci,
\emph{Linear theory for a mixed operator with Neumann conditions},
Asymptot. Anal. \textbf{128} (2022), 571--594.

\bibitem{DPV2023}
S. Dipierro, E. Proietti Lippi, and E. Valdinoci,
\emph{(Non)local logistic equations with Neumann conditions},
Ann. Inst. H. Poincar\'e C Anal. Non Lin\'eaire \textbf{40} (2023), 1093--1166.

\bibitem{DSVZ2025}
S. Dipierro, X. Su, E. Valdinoci, and J. Zhang,
\emph{Qualitative properties of positive solutions of a mixed order nonlinear Schr\"odinger equation},
Discrete Contin. Dyn. Syst. \textbf{45} (2025), 1948--2000.

\bibitem{DV2021}
S. Dipierro and E. Valdinoci,
\emph{Description of an ecological niche for a mixed local/nonlocal dispersal. An evolution equation and a new Neumann condition arising from the superposition of Brownian and L\'evy processes},
Phys. A \textbf{575} (2021), Article 126052.

\bibitem{DRZ2021}
R. Dumitrescu, C. Reisinger, and Y. Zhang,
\emph{Approximation schemes for mixed optimal stopping and control problems with nonlinear expectations and jumps},
Appl. Math. Optim. \textbf{83} (2021), 1387--1429.

\bibitem{FQ2011}
P. Felmer and A. Quaas,
\emph{Fundamental solutions and Liouville type theorems for nonlinear integral operators},
Adv. Math. \textbf{226} (2011), 2712--2738.

\bibitem{FR2024}
X. Fern\'andez-Real and X. Ros-Oton,
\emph{Integro-Differential Elliptic Equations},
Progress in Mathematics, vol. 350, Birkh\"auser, Cham, 2024.

\bibitem{GK2022}
P. Garain and J. Kinnunen,
\emph{On the regularity theory for mixed local and nonlocal quasilinear elliptic equations},
Trans. Amer. Math. Soc. \textbf{375} (2022), 5393--5423.

\bibitem{GL2023}
P. Garain and E. Lindgren,
\emph{Higher H\"older regularity for mixed local and nonlocal degenerate elliptic equations},
Calc. Var. Partial Differential Equations \textbf{62} (2023), Article 67.


\bibitem{GilbargSerrin1955}
D. Gilbarg and J. Serrin,
\emph{On isolated singularities of solutions of second order elliptic differential equations},
J. Analyse Math. \textbf{4} (1955/56), 309--340.


\bibitem{Gidas1980}
B. Gidas,
\emph{Symmetry properties and isolated singularities of positive solutions of nonlinear elliptic equations},
in Nonlinear Partial Differential Equations in Engineering and Applied Science, Lecture Notes in Pure and Appl. Math., vol. 54, Dekker, New York, 1980, 255--273.

\bibitem{GS1981}
B. Gidas and J. Spruck,
\emph{Global and local behavior of positive solutions of nonlinear elliptic equations},
Comm. Pure Appl. Math. \textbf{34} (1981), 525--598.

\bibitem{GL1984}
F. Gimbert and P.-L. Lions,
\emph{Existence and regularity results for solutions of second-order elliptic integro-differential operators},
Ricerche Mat. \textbf{33} (1984), 315--358.

\bibitem{GLX2026}
Y. Guo, C. Li, and J. Xie,
\emph{Liouville theorems for the Lane--Emden equation involving a mixed local--nonlocal operator},
arXiv preprint 2606.16940, 2026.

\bibitem{HS2026}
H. Hajaiej and Y. Su,
\emph{The best constant in the G-N inequality for the mixed local and nonlocal Laplacian},
arXiv preprint 2604.05177, 2026.




\bibitem{HW1954}
P. Hartman and A. Wintner,
\emph{On the local behavior of solutions of non-parabolic partial differential equations. II. The uniqueness of the Green singularity},
Amer. J. Math. \textbf{76} (1954), 351--361.


\bibitem{JK2006}
E.~R. Jakobsen and K.~H. Karlsen,
\emph{A maximum principle for semicontinuous functions applicable to integro-partial differential equations},
NoDEA Nonlinear Differential Equations Appl. \textbf{13} (2006), 137--165.

\bibitem{K2026}
T. Klimsiak,
\emph{B\^ocher type theorem for elliptic equations with drift-perturbed L\'evy operators},
J. Math. Pures Appl. \textbf{213} (2026), Article 103939.

\bibitem{LLWX2020}
C. Li, C. Liu, Z. Wu, and H. Xu,
\emph{Non-negative solutions to fractional Laplace equations with isolated singularity},
Adv. Math. \textbf{373} (2020), Article 107329.

\bibitem{LWX2018}
C. Li, Z. Wu, and H. Xu,
\emph{Maximum principles and B\^ocher type theorems},
Proc. Natl. Acad. Sci. USA \textbf{115} (2018), 6976--6979.


\bibitem{RaoSongVondracek2006}
M.~Rao, R.~Song, and Z.~Vondra\v{c}ek,
\emph{Green function estimates and Harnack inequality for subordinate
Brownian motions}, Potential Anal. \textbf{25} (2006), 1--27.

\bibitem{S2007}
L. Silvestre,
\emph{Regularity of the obstacle problem for a fractional power of the Laplace operator},
Comm. Pure Appl. Math. \textbf{60} (2007), 67--112.

\bibitem{SV2007}
R. Song and Z. Vondra\v cek,
\emph{Parabolic Harnack inequality for the mixture of Brownian motion and stable process},
Tohoku Math. J. (2) \textbf{59} (2007), 1--19.

\bibitem{SVWZ2022}
X. Su, E. Valdinoci, Y. Wei, and J. Zhang,
\emph{Regularity results for solutions of mixed local and nonlocal elliptic equations},
Math. Z. \textbf{302} (2022), 1855--1878.

\bibitem{SVWZ2025}
X. Su, E. Valdinoci, Y. Wei, and J. Zhang,
\emph{On some regularity properties of mixed local and nonlocal elliptic equations},
J. Differential Equations \textbf{416} (2025), 576--613.

\bibitem{SX2025}
X. Su and S. Xu,
\emph{Qualitative properties of positive solutions to mixed local and nonlocal critical problems in $\mathbb R^n$},
arXiv preprint 2512.21873, 2025.

\end{thebibliography}
\end{document}